\documentclass[11pt]{amsart}
\usepackage{epsf}
\usepackage{amsbsy,amsmath}
\usepackage{mathtools}
\usepackage{mathrsfs}
\usepackage{amsfonts}
\usepackage{amssymb}
\usepackage{enumitem}
\usepackage{eucal}
\usepackage{enumitem}
\usepackage{graphics,mathrsfs}
\usepackage{amsthm}
\usepackage{xcolor}
\newtheorem{theorem}{Theorem}[section]
\newtheorem{lemma}[theorem]{Lemma}

\theoremstyle{definition}
\newtheorem{definition}[theorem]{Definition}

\theoremstyle{remark}
\theoremstyle{claim}
\newtheorem{remark}[theorem]{Remark}
\numberwithin{equation}{section}

\numberwithin{equation}{section}

\newsavebox{\savepar}

\begin{document}
%	\begin{frontmatter}		
		\title[A fractional $p$-Kirchhoff problem]{Existence of at least $k$ solutions to a fractional $p$-Kirchhoff problem involving singularity and critical exponent}
		
\author[S. Ghosh]{Sekhar Ghosh}
\address[S. Ghosh]{Department of Mathematics, National Institute of Technology Calicut, Calicut, 673601, Kerala, India}
\email{sekharghosh1234@gmail.com, sekharghosh@nitc.ac.in}

\author[D. Choudhuri]{Debajyoti Choudhuri}
\address[D. Choudhuri]{School of Basic Sciences, Indian Institute of Technology Bhubaneswar, Khordha - 752050, Odisha, India}
\email{dc.iit12@gmail.com, dchoudhuri@iitbbs.ac.in}

\author[A. Fiscella]{Alessio Fiscella}
\address[A. Fiscella]{Dipartimento di Matematica e Applicazioni, Universit\`a degli Studi di Milano-Bicocca, Via Cozzi 55, Milano, 20125, Italy}
\email{alessio.fiscella@unimib.it}

\maketitle
	
		\begin{abstract}
			We study the existence of nonnegative solutions to the following nonlocal elliptic problem involving singularity
			\begin{align}
				\mathfrak{M}\left(\int_{Q}\frac{|u(x)-u(y)|^p}{|x-y|^{N+ps}}dxdy\right)(-\Delta)_{p}^{s} u&=\frac{\lambda}{|u|^{\gamma-1}u}+|u|^{p_s^*-2}u~\text{in}~\Omega,\nonumber\\
				u&>0~\text{in}~\Omega,\nonumber\\
				u&=0~\text{in}~\mathbb{R}^N\setminus\Omega,\nonumber
			\end{align}			
			where $\Omega\subset\mathbb{R}^N$, is a bounded domain with Lipschitz boundary, $\lambda>0$, $N>ps$, $0<s,\gamma<1$, $(-\Delta)_{p}^{s}$ is the fractional $p$-Laplacian operator for $1<p<\infty$ and $p_s^*=\frac{Np}{N-ps}$ is the critical Sobolev exponent. We employ a {\it cut-off} argument to obtain the existence of $k$ (being arbitrarily large integer) solutions. Furthermore, by using the Moser iteration technique, we prove an uniform $L^{\infty}({\Omega})$ bound for the solutions. The novelty of this work lies in proving the existence of small energy solutions by using symmetric mountain pass theorem in spite of the presence of a critical nonlinear term which, of course, is super-linear.
		\end{abstract}
		\begin{flushleft}
{\bf Keywords}:~Fractional $p$-Laplacian, Critical Exponent, Concentration-Compactness Principle, Genus, Symmetric Mountain Pass Theorem, Singularity.\\
			%MSC codes here, in the form: 
			{\bf AMS Classification}:~ 35R11, 35J60, 35J75.
			%% or \MSC[2008] code \sep code (2000 is the default)46E35, 46B50
			
		\end{flushleft}
		
	%\end{frontmatter}
	
	%\linenumbers

	\section{Introduction}	
	\noindent This paper aims to study the following nonlocal Kirchhoff-type elliptic problem involving singularity and critical exponent
	\begin{align*}\label{main p}	
	\mathfrak{M}\left(\int_{Q}\frac{|u(x)-u(y)|^p}{|x-y|^{N+ps}}dxdy\right)(-\Delta)_{p}^{s} u&=\frac{\lambda}{|u|^{\gamma-1}u}+|u|^{p_s^*-2}u~\text{in}~\Omega,\\
	u&>0~\text{in}~\Omega,\\
	u&=0~\text{in}~\mathbb{R}^N\setminus\Omega,\tag{P}
	\end{align*}
	where $\Omega\subset\mathbb{R}^N$ is a bounded domain with Lipschitz boundary, $\lambda>0$, $N>ps$, $0<s, \gamma<1$ and $p_s^*=\frac{Np}{N-ps}$ is the critical Sobolev exponent. The Kirchhoff function $\mathfrak{M}$ is supposed to satisfy the following conditions:
	\begin{itemize}
		\item[$(\mathfrak{m}_1)$] The function $\mathfrak{M}:\mathbb{R}^+\rightarrow\mathbb{R}^+$ is continuous and there exists $\theta\in\left(1,\frac{p_s^*}{p}\right)$ such that $t\mathfrak{M}(t)\leq\theta\mathcal{M}(t)$ for all $t\geq 0$, where $\mathcal{M}(t)=\int_0^t\mathfrak{M}(\tau)d\tau$.
		%\item[$(\mathfrak{m}_2)$]  $\underset{t\geq0}{\inf}\{\mathfrak{M}(t)\}=\mathfrak{m}_0>0$.
		\item [$(\mathfrak{m}_2)$] There exists $\mathfrak{m}_0>0$ such that $\mathcal{M}(t)\geq \mathfrak{m}_0t^{\theta-1}$ for $t\in[0,1]$.
		\item  [$(\mathfrak{m}_3)$] For any $\tau>0$, there exists $\kappa=\kappa(\tau)>0$ such that $\mathfrak{M}(t)\geq\kappa$ for all $t\geq\tau$.
	\end{itemize}
	Here, $\mathfrak{M}$ is a degenerate Kirchhoff function if $\mathfrak{M}(0)=0$, otherwise $\mathfrak{M}$ is said to be a non-degenerate Kirchhoff function. We define the fractional $p$-Laplacian as follows.
	\begin{eqnarray}
		(-\Delta)_p^su(x)=C_{N,s}\lim_{\varepsilon\rightarrow 0}\int_{\mathbb{R}^N\setminus B_{\varepsilon}(x)}\frac{|u(x)-u(y)|^{p-2}(u(x)-u(y))}{|x-y|^{N+sp}}dy,~\text{in}~x\in\mathbb{R}^N.\nonumber
	\end{eqnarray}
	Recently a lot of attention has been paid to the elliptic problems involving local/nonlocal operators with singularities. These works are not only important from the mathematical point of view but also have lots of applications, for instance in thin obstacle problems \cite{Silvestre2007}, problems on minimal surfaces \cite{Caffarelli2010}, fractional quantum mechanics \cite{Laskin2000} etc.(refer \cite{Nezza2012, Bisci2016, Papageorgiou2019} and the references therein).	``{\it The problem draws its motivation from the models presented by Kirchhoff in 1883 as a generalization of the D'Alembert wave equation.
		$$\rho\frac{\partial^2u}{\partial t^2}-\left(a+b\int_{0}^{l}\left|\frac{\partial u}{\partial x}\right|^2dx\right)\frac{\partial^2u}{\partial x^2}=g(x,u)$$
	where $a, b, \rho$ are positive constants and $l$ is the changes in the length of the strings due to the vibrations.}" An elaborate detailing on these for the fractional counterpart can be found in \cite[Appendix A]{Fiscella2014}, \cite{Pucci2016} and the references therein. For further details on practical applications, one may refer \cite{Alves2005,Caffarelli2012, Bisci2014} and the references therein. Elliptic problems with a singularity have not only been important but also a tough challenge to the mathematical community. The roots of the problem can be traced back to a celebrated work due to Lazer and McKenna \cite{Lazer1991}, where the authors considered the following problem
	\begin{eqnarray}
		-\Delta u&=&p(x)u^{-\gamma},~\text{in}~\Omega\nonumber\\
		u&=&0,~\text{on}~\partial\Omega\nonumber
	\end{eqnarray}
	where $\Omega\subset\mathbb{R}^N$ is a sufficiently regular domain. Also $p$ is
	a sufficiently regular function which is positive in $\overline{\Omega}$. The solution $u$ is in $W_0^{1,2}(\Omega)$ if and only if $\gamma<3$. The authors in \cite{Lazer1991} proved that if $\gamma >1$, then $u$ is not in $C^1(\overline{\Omega})$ whereas for $0<\gamma<1$ the solution obtained is a classical solution. Thereafter a lot of work on elliptic problems involving a singularity has also been considered whose existence and multiplicity results have been investigated. As a passing list of references on some pioneering study of problems involving a purely singular term one may refer to \cite{Boccardo2009,Canino2017,Crandall1977} and the references therein. With the development of newer tools in functional analysis, the problems with a singularity also became richer.  One such instance is a problem investigated by Giacomoni et al.  \cite{Giacomoni2007}. The problem is as follows
	\begin{eqnarray}\label{giaco}
		-\Delta_p u&=&\lambda u^{-\gamma}+u^q,~\text{in}~\Omega\nonumber\\
		u&=&0,~\text{on}~\partial\Omega\nonumber\nonumber\\
		u&>&0,~\text{in}~\Omega
	\end{eqnarray}
	where $1<p-1<q\leq p^*-1$, $\lambda>0$, $0<\gamma<1$. Here the authors have proved the existence of two positive solutions. Similar type of results to obtain existence and multiplicity (finitely many) of solutions can be found in \cite{Giacomoni2009, Haitao2003,Mukherjee2016,Saoudi2017,Saoudi2019} and the references therein. Recently, Saoudi et al. \cite{Saoudi2019} considered a fractional $p$-Laplacian version of \eqref{giaco} and proved the existence of two solutions to it by using a variational methods. Besides this the authors in \cite{Saoudi2019} used Moser's iteration method to prove that the solutions are in $L^{\infty}$. A $W_0^{s,p}$ versus $C^1$ analysis has also been discussed in it.\\
	We now focus on some of the work which bears the Kirchhoff term, with or without singularities. Moving on from here, we now turn our attention to problems involving a critical exponent. If one considers $\mathfrak{M}=1$, $\lambda=0$ in \eqref{main p}, then the problem reduces to the following
	\begin{align}\label{crit_prob1}
		(-\Delta)_p^s u&=|u|^{p_s^*-2}u,~\text{in}~\Omega\nonumber\\
		u&=0,~\text{on}~\mathbb{R}^N\setminus\partial\Omega.
	\end{align}
	The main hurdle with problems with critical exponent is the lack of compact embedding $W_0^{s,p}(\Omega)\hookrightarrow L^{p^*_s}(\Omega)$. Such problems are tackled by the concentration-compactness principle introduced by Lions \cite{Lions1985, Lions1985a} for the classical Sobolev spaces $W^{1,p}(\Omega)$. The nonlocal version of this principle has also been developed and can be found in \cite[Theorem 2.2]{Xiang2017}. The literature pertaining to these type of problems without singularity are so vast that it can't be discussed here in this section completely. However, the readers may refer to the books \cite{Bisci2016,Papageorgiou2019} and the references therein. Of late, existence and multiplicity of solutions to the Kirchhoff problem \eqref{main p} have been investigated by many researchers. The reader may refer to \cite{Fiscella2019,Fiscellapk2019,Hsini2019,Wang2021} and the references therein. In \cite{Fiscellapk2019,Hsini2019}, the authors have established the existence of at least two solutions by employing the Nehari manifold method. The authors in \cite{Wang2021} investigated the existence of two solutions together with a Choquard term of a problem of the type \eqref{main p}. Fiscella \cite{Fiscella2019} has employed the variational method in combination with a perturbation method to guarantee the existence of two solutions to the problem \eqref{main p} even if the Kirchhoff function is degenerate. It is important to note that in all these studies the authors guaranteed the existence of two solutions by employing different variational tools. In this article, we establish the existence of arbitrarily many small solutions to the problem \eqref{main p} by employing symmetric mountain pass theorem. The symmetric mountain pass theorem is mainly used to guarantee the existence of infinitely many solutions. Usually it is necessary to apply the symmetric mountain pass theorem to an elliptic PDE that the corresponding to a $C^1$ energy. This requirement fails to hold due to the singularity. We have tackled this delicate issue by employing a cut-off technique.
	
	One of the earliest study for the existence of infinitely many solutions to problems with critical exponent is due to the pioneering work by Azorero and Alosonso \cite{Garcia1991}, who have considered the following problem
	\begin{align}\label{crit_prob2}
		-\Delta_p u&=|u|^{p^*-2}u+\lambda |u|^{q-2}u,~\text{in}~\Omega\nonumber\\
		u&=0,~\text{on}~\partial\Omega
	\end{align}
	for $1<q<p$, $\lambda>0$. Here the authors have used the Lusternik-Schnirelman's theory to guarantee the existence of infinitely many solutions. The problem in \cite{Garcia1991} was further generalized by Li and Zhang \cite{Li2009} with the driving operator being $-\Delta_p-\Delta_q$. The reader may also refer to the work due to Figueiredo \cite{Figueiredo2013}. We now throw some light on the $p$-Kirchhoff problems of the following type that has been discussed in Khiddi and Sbai \cite{Khiddi2020}
	\begin{align*}
		\mathfrak{M}\left(\iint_{\mathbb{R}^N}\frac{|u(x)-u(y)|^p}{|x-y|^{N+sp}}dxdy\right)(-\Delta)_p^su&=\lambda H(x)|u|^{q-2}u+ |u|^{p_s^*-2}u,~\text{in}~\Omega\\
		u&=0,~\text{in}~\mathbb{R}^N\setminus\Omega,
	\end{align*}
	where $\lambda>0$, $1<q<p<p_s^*<\infty$. The authors in \cite{Khiddi2020} have guaranteed the existence of infinitely many solutions. Such type of problems have led to the generalization of a few classical results for the case of $\mathfrak{M}=1.$ In \cite{Xiang2015a}, the authors employed the Fountain and the dual Fountain theorem to guarantee the existence of infinitely many solutions for a symmetric subcritical Kirchhoff problem for a non-degenerate $\mathfrak{M}$ and $p\geq 2$. In \cite{Fiscella2016}, the authors dealt with the case when $p=2$ and used the notion of Krasnoselskii's genus (refer \cite{Rabinowitz1986}) to obtain the existence of infinitely many solutions. Further in \cite{Xiang2016} the authors had a similar conclusion but for a system of PDEs with subcritical degenerate Kirchhoff function. This is in no way a complete picture of the literature developed so far as it is vast. What we can do at this point is to direct the attention of the reader to the problem which prompted us to take up \eqref{main p}. The motivation of this problem was drawn from the results due to Azorero et al. \cite{Garcia1991}, Khiddi-Sbai \cite{Khiddi2020}. The literature consisting the study of infinitely many solution mainly deals with the concave-convex data, which may be both sub-linear as well as super-linear.  Recently, in its first kind the study due to \cite{Ghosh2019} guarantees the existence of infinitely solutions involving a singularity.
	
	Motivated from the above studies, in this article, we prove that problem \eqref{main p} possesses at least $k\in\mathbb{N}$ solutions (for arbitrarily large $k$) within a finite range of $\lambda$ whose space norms converges to zero. It is worthy to mention here that the symmetric mountain pass theorem plays a key role to study the existence of infinitely many solutions to a PDE. The symmetric mountain pass theorem has two type of conclusions consisting a sequence of solutions. One is for sub-linear data in which the space norm of the solutions converges to zero another one is for the super-linear data which says the space norm of solutions goes to infinity. The major hurdles to us were to figure out a way to tackle the singular term as well as the critical exponent term, which is super-linear in the problem \eqref{main p} and then to show that as $k$ increases, the space norm of solutions decreases toward zero. To add to these issues, the functional also fails to be coercive. The main result proved in this article is the following.
\begin{theorem}\label{thm main}
	Let $\mathfrak{m}_1$-$\mathfrak{m}_2$ hold and $0<\gamma<1$. Then for any $k\in\mathbb{N}$ (arbitrarily large), there exists $\lambda_*>0$ such that whenever $0<\lambda<\lambda_*$,  problem \eqref{main p} has at least $k$ non-negative weak solutions $\{u_1,u_2,\ldots, u_k,\ldots\}$ such that $J_{\lambda}(u_n)<0$, for all $n=1,2,\ldots,k,\ldots$ In addition, as $k$ increases, then the norms of $J_{\lambda}(u_k)$ and $u_k$ decreases.  Furthermore, each solution of \eqref{main p} belongs to $L^{\infty}(\bar{\Omega})$.
\end{theorem}
\begin{remark}
	 The conclusion of the Theorem \ref{thm main} holds true even if we consider subcritical but super-linear exponent instead of $p_s^*$.
		%\item It will be interesting to study whether the conclusion of Theorem \ref{thm main} holds true if we take a degenerate Kirchhoff function, i.e. if $\mathfrak{m}_0=0$.
\end{remark}
\section{Preliminaries and Weak Formulations}
\noindent In this section we will first recall some properties of the fractional Sobolev spaces. Let $\Omega$ be a bounded domain in $\mathbb{R}^N, N\geq2$ with Lipschitz boundary and define $Q=\mathbb{R}^{2N}\setminus((\mathbb{R}^N\setminus\Omega)\times(\mathbb{R}^N\setminus\Omega))$. Consider the Banach space $(X, \|\cdot\|_X)$ such that
\begin{eqnarray}
	X=\left\{u:\mathbb{R}^N\rightarrow\mathbb{R}\,\,\text{is measurable},\,\, u|_{\Omega}\in L^p(\Omega) \,\,\text{and}\,\,\frac{|u(x)-u(y)|}{|x-y|^{\frac{N+ps}{p}}}\in L^{p}(Q)\right\}
\end{eqnarray}
with respect to the well known Gagliardo norm
\begin{eqnarray}
	\|u\|_X=\|u\|_{L^p(\Omega)}+\left(\int_{Q}\frac{|u(x)-u(y)|^p}{|x-y|^{N+ps}}dxdy\right)^{\frac{1}{p}}.\nonumber
\end{eqnarray}
Let $X_0$ be the subspace of $X$ defined as 
\begin{eqnarray}
	X_0=\left\{u\in X:\,\, u=0 \,\,\text{a.e. in}\,\, \mathbb{R}^N\setminus\Omega\right\}.\nonumber
\end{eqnarray}
Then the space $(X_0, \|\cdot\|)$ is a Banach space \cite{Servadei2013,Servadei2012} with respect to the norm
\begin{eqnarray}
	\|u\|=\left(\int_{Q}\frac{|u(x)-u(y)|^p}{|x-y|^{N+ps}}dxdy\right)^{\frac{1}{p}}.\nonumber
\end{eqnarray}
\noindent The best Sobolev constant is defined as 
\begin{equation}\label{sobolev const}
	S=\underset{u\in X_0\setminus\{0\}}{\inf}\displaystyle\cfrac{\int_{Q}\frac{|u(x)-u(y)|^p}{|x-y|^{N+ps}}dxdy}{\left(\int_\Omega|u|^{p_s^*}dx\right)^{\frac{p}{p_s^*}}}.
\end{equation}
\noindent The next Lemma is due to \cite{Servadei2013,Servadei2012} which states the embeddings of the space $X_0$ into Lebesgue spaces.
\begin{lemma}\label{embedding}
	If $\Omega$ is a bounded domain with a Lipschitz boundary and $N>ps$, then the embedding $X_0 \hookrightarrow L^{q}(\Omega) $ for $q\in[1,p_s^*]$ is continuous and is compact for $q\in[1,p_s^*)$, where $p_s^*=\frac{Np}{N-ps}$ is the critical Sobolev exponent.
\end{lemma}
\noindent A function $u\in X_0$ is a weak solution of problem (\ref{main p}), if $\varphi u^{-\gamma}\in L^1(\Omega)$ and
\begin{align}\label{weak p}
\begin{split}
 \mathfrak{M}(\|u\|^p)\int_{Q}\frac{|u(x)-u(y)|^{p-2}(u(x)-u(y))(\varphi(x)-\varphi(y))}{|x-y|^{N+ps}}dxdy &-\int_{\Omega}\frac{\lambda}{|u|^{\gamma-1}u}\varphi dx\\
 &-\int_{\Omega}|u|^{p_s^*-2}u\varphi dx=0
 \end{split}
\end{align}
for each $\varphi\in X_0$. The energy functional $J_\lambda\colon X_0\rightarrow\mathbb{R}$ associated to problem \eqref{main p} is defined as
\begin{align}\label{energy p}
J_{\lambda}(u) = 
\frac{1}{p}\mathcal{M}(\|u\|^p)\|u\|^p-\frac{\lambda}{1-\gamma}\int_{\Omega}|u|^{1-\gamma}dx-\frac{1}{p_s^*}\int_{\Omega}|u|^{p_s^*}dx.
\end{align}
Observe that because of the presence of a singular term, the functional $J_{\lambda}$ is not differentiable in $X_0$. Also, the critical Sobolev nonlinearity causes the lack of compact embedding (refer Lemma \ref{embedding}). Hence we will use a cut-off technique to accomplish our goal. Precisely, we will first construct a $C^1$ functional $I_{\lambda}$ such that the critical points of both $J_{\lambda}$ and $I_{\lambda}$ coincides whenever the norms of the critical points are sufficiently small. Then on applying the concentration compactness principle, we will obtain a certain range of {\it energy} for the functional to satisfy the Palais-Smale condition. Finally, by using the Kajikiya's symmetric mountain pass theorem \cite{Kajikiya2005} the existence of $k$ solutions will be achieved. One can prove the G\^{a}teaux differentiability of $J_{\lambda}$ with a slight modification of cf.\cite[Lemma 6.2]{Saoudi2019}.

Let us now consider the singular Kirchhoff problem
\begin{align}\label{squasinna}
\begin{split}
	\mathfrak{M}\left(\int_{Q}\frac{|u(x)-u(y)|^p}{|x-y|^{N+ps}}dxdy\right)(-\Delta)_{p}^{s} u&=\lambda u^{-\gamma},~\text{in}~\Omega\\
	u&>0,~\text{in}~\Omega\\
	u&=0,~\text{in}~\mathbb{R}^N\setminus\Omega.
	\end{split}
\end{align}
 The existence of a weak solution of \eqref{squasinna} can be obtained as a solution for limits of approximations by considering a sequence of perturbed PDEs. We now have the following Lemma due to \cite{Canino2017}.
\begin{lemma}
	Let $\Omega\subset\mathbb{R}^N$ be a bounded domain with Lipschitz boundary. Then, for $0<\gamma<1$, problem \eqref{squasinna} possesses a unique solution $\underline{u}\in X_0$ with $\underset{K}{\mbox{\upshape ess\,inf}}\,\underline{u}>0$ for any compact $K\subset\subset\Omega$.
\end{lemma}
\noindent Inspired from Giacomoni et al \cite{Giacomoni2007}, we now consider the following cut-off problem
\begin{align*}\label{cut p}	
	\mathfrak{M}\left(\int_{Q}\frac{|u(x)-u(y)|^p}{|x-y|^{N+ps}}dxdy\right)(-\Delta)_{p}^{s} u&=\lambda g(x,u)+|u|^{p^*_s-2}u~\text{in}~\Omega,\\
	u&>0~\text{in}~\Omega,\\
	u&=0~\text{in}~\mathbb{R}^N\setminus\Omega,\tag{$P'$}
\end{align*}
\noindent where the function $g$ is defined as below. Let $\underline{u}$ be the unique solution to \eqref{squasinna}. We define $g:\Omega\times\mathbb{R}^+\to\mathbb{R}$ as
\[   
{g}(x,t) = 
\begin{cases}
|t|^{-\gamma-1}t, &~\text{if}~|t|>\underline{u}(x)\\
(\underline{u}(x))^{-\gamma},&~\text{if}~ 0<|t|\leq \underline{u}(x)
%0, &~\text{if}~t\leq 0
\end{cases}\] 
%and
%\[   
%{f}(x,t) = 
%\begin{cases}
%t^{p_s^*-1}, &~\text{if}~t>\underline{u}(x)\\
%(\underline{u}(x))^{p_s^*-1},&~\text{if}~t\leq \underline{u}(x).
%\end{cases}\]
Let $G(x,u)=\int_{0}^{u}{g}(x,\tau)d\tau$.
%and $F(x,t)=\int_{0}^{t}{f}(x,\alpha)d\alpha$. 
Therefore, the associated cut-off functional $I_\lambda\colon X_0\rightarrow\mathbb{R}$ corresponding to \eqref{cut p} is defined as
\begin{align*}
I_{\lambda}(u) = 
\frac{1}{p}\mathcal{M}(\|u\|^p)-\lambda\int_{\Omega}G(x,u)dx-\frac{1}{p_s^*}\int_{\Omega}|u|^{p_s^*}dx.
\end{align*}
Note that the functional $I_{\lambda}$ is even. Moreover, it can be concluded that $I_\lambda$ is $C^1$. Indeed, the first and the third term of the functional $I_{\lambda}$ are $C^1$. On applying a similar arguments as in \cite[Lemma 6.4]{Saoudi2019} {in combination with Remark \ref{dj1} $(2)$} we get the differentiability of the singular term $G(x,u)$ for nonnegative $u$.
%The functional $J_{\lambda}$ is G\^{a}teaux differentiable.

% We consider the following problem 
%\begin{align}\label{min_prob1}
%\underset{u\in B_r}\inf\{I_{\lambda}(u)\}.
%\end{align}

%, then the set of critical points of $I_{\lambda}$ coincides with that of $J_{\lambda}$, over the set $\{x\in\Omega:|u(x)|>\underline{u}(x)\}$.
We aim to guarantee that the problem \eqref{cut p} possesses arbitrarily (but countably) many nonnegative solutions depending upon the value of $\lambda$.
A function $u\in X_0$ is said to be a weak solution to the problem (\ref{cut p}), if $\varphi u^{-\gamma}\in L^1(\Omega)$ and 
\begin{align*}
\begin{split}
\mathfrak{M}(\|u\|^p)\int_{Q}\frac{|u(x)-u(y)|^{p-2}(u(x)-u(y))(\varphi(x)-\varphi(y))}{|x-y|^{N+ps}}dxdy&-\lambda\int_{\Omega}g(x,u)\varphi dx\\
&-\int_{\Omega}|u|^{p^*-2}u\varphi dx=0
\end{split}
\end{align*}
for each $\varphi\in X_0$.

{We apriori prove that if $u>0$ is a solution to \eqref{cut p} then $u\geq \underline{u}$ a.e. in $\Omega$, so that the critical points of the functionals $I_{\lambda}$ and $J_{\lambda}$ become identical.} Indeed, let $u$ be a positive solution to \eqref{main p} and $A=\{x\in\Omega: u(x)\leq \underline{u}(x)\}$. We first consider the difference of the problems \eqref{main p} and \eqref{squasinna}, and test it with $\phi=(u-\underline{u})_-$. We then use the fact that the Kirchhoff operator is monotone in nature (proof follows from Bensedik \cite{ahmed1})
\begin{align}\label{dj1}
\begin{split}
0\geq \langle \mathfrak{M}(\|u\|^p)(-\Delta)_p^su-\mathfrak{M}(\|\underline{u}\|^p)(-\Delta)_p^s\underline{u},(u-\underline{u})_-\rangle=&\lambda\int_{\Omega}(u^{-\gamma}-\underline{u}^{-\gamma})(u-\underline{u})_-dx\\
&+\int_{\Omega}u^{p_s^*-1}(u-\underline{u})_-dx\geq 0.
\end{split}
\end{align}
Thus $|A|=0$. {Therefore, $u\geq \underline{u}$ a.e. in $\Omega$.}

	\begin{lemma}\label{u_greater_u_lambda}
		Let $\lambda$ be sufficiently small. Then any positive solution of \eqref{main p}, say $u$, is such that $u\geq \underline{u}$, then $u>\underline{u}$ a.e. in $\Omega$.
	\end{lemma}
	\begin{proof}
		Apparently $\underline{u}$ is a subsolution to the positive solutions of \eqref{main p} indicating that $u\geq \underline{u}$ a.e. in $\Omega$. Let $B=\{x\in\Omega:u(x)=\underline{u}(x)\}$. This set $B$ is measurable and therefore for any $\delta>0$ there exists a closed subset $F$ of $B$ such that $|B\setminus F|<\delta$. Also, let $|B|>0$. Define a test function $\varphi\in C_c^1(\Omega)$ such that 
		\begin{equation}\varphi(x)=\begin{cases}
		1,& ~\text{if}~ x\in F\\
		0<\varphi<1,&~\text{if}~x\in B\setminus F\\
		0,&~\text{if}~x\in \Omega\setminus B.
		\end{cases}\end{equation}
		Since $u$ is a weak solution to \eqref{main p}, we have 
		\begin{align}\label{equality_breakage}
		\begin{split}
		0=&  \langle \mathfrak{M}(\|u\|^p)(-\Delta)_p^su,\varphi\rangle-\lambda\int_{F}u^{-\gamma}dx-\lambda\int_{B\setminus F}u^{-\gamma}\varphi dx-\int_{F}u^{p_s^*-1}dx-\int_{B\setminus F}u^{p_s^*-1}\varphi dx\\
		=&-\int_{F}u^{p_s^*-1}dx-\int_{B\setminus F}u^{p_s^*-1}\varphi dx<0
		\end{split}
		\end{align}
		which is absurd. So, $|B|=0$. Hence, $u>\underline{u}$ a.e. in $\Omega$.
	\end{proof}
	\noindent Therefore we choose $\underline{u}:=u_0$ in the cut-off functional.

From this, the weak solutions of problem \eqref{cut p} are precisely the critical points of the energy functional $I_{\lambda}$. We note that the critical points of $I_{\lambda}$ are also weak solutions to the problem (\ref{cut p}).

\section{Sufficient energy range for the Palais-Smale Condition.}
\noindent The aim of this section is to see whether the functional $I_{\lambda}$ satisfies the Palais-Smale $(PS)_c$ condition up to some finite energy level $c$ or not. We also target to apply the symmetric mountain pass theorem to the functional $I_{\lambda}$. Observe that due to the presence of the critical exponent, the functional $I_{\lambda}$ is not bounded from below. Moreover, due to the lack of a compact embedding of $X_0$ in $L^{p^*_s}(\Omega)$, one cannot verify the $(PS)_c$ condition instantly. However, we will look for an energy range $(-\infty,c_*)$ such that the $(PS)_c$ condition holds true. For this, we first state the $(PS)_c$ condition for $I_{\lambda}$.
\begin{definition}[$(PS)_c$ condition for $I_{\lambda}$] Let $c\in \mathbb{R}$ and let $X_0^{*}$ be the dual of $X_0$. Then $\{u_{n}\}\subset X_0$ is said to be a $(PS)_c$ sequence for $I_{\lambda}$ if it verifies
\begin{align}\label{ps 1}
	I_{\lambda}(u_{n})\rightarrow c\quad\text{and}\quad\langle I_{\lambda}^{'}(u_{n}), v\rangle\rightarrow0,\quad \text{for every } v\in X_0.
\end{align}
Moreover, $I_{\lambda}$ satisfies the $(PS)_{c}$ condition if every $(PS)_{c}$ sequence for $I_{\lambda}$ possesses a convergent subsequence.
\end{definition}

Now, let us set
\begin{align}\label{cond1}
c_*:=&\min\left\{\left(\frac{1}{p\theta}-\frac{1}{p_s^*}\right)(\kappa S)^{p_s^*/(p_s^*-p)},\left(\frac{1}{p\theta}-\frac{1}{p_s^*}\right)(m_0^{1/\theta}S)^{\frac{p_s^*\theta}{p_s^*-p\theta}}\right\}\nonumber\\
&-\left(\frac{1}{p\theta}-\frac{1}{p_s^*}\right)^{-\frac{1-\gamma}{p_s^*-1+\gamma}}\left[\lambda \left(\frac{1}{1-\gamma}-\frac{1}{p\theta}\right)C(\Omega)\right] ^{\frac{p_s^*}{p_s^*-1+\gamma}}\\
&-\lambda\left(\frac{1}{1-\gamma}-\frac{1}{p\theta}\right)\int_{\Omega}(\underline{u})^{1-\gamma},
\end{align}	
where $\kappa=\kappa(1)$ and $\mathfrak{m}_0$ are the constants from the conditions $(\mathfrak{m}_2)$ and $(\mathfrak{m}_3)$, $\underline{u}$ is the unique solution of \eqref{squasinna} and we denote
$$C(\Omega)=|\Omega|^{(p_s^*-1+\gamma)/p_s^*}S^{-(1-\gamma)/p}$$
with $S$ given in \eqref{sobolev const}.

\begin{lemma}\label{ps limit}
	Let the assumptions $(\mathfrak{m}_{1})$-$(\mathfrak{m}_{2})$ hold true and let $\lambda>0$. Then, the functional $I_{\lambda}$ satisfies $(PS)_c$-condition for any $c<c_*$, with $c_*$ given in \eqref{cond1}.
%	for every 
%	\begin{align*}c<c_*:=&\min\left\{\left(\frac{1}{p\theta}-\frac{1}{p_s^*}\right)(\kappa S)^{p_s^*/(p_s^*-p)},\left(\frac{1}{p\theta}-\frac{1}{p_s^*}\right)(m_0^{1/\theta}S)^{\frac{p_s^*\theta}{p_s^*-p\theta}}\right\}\\
%	&-\left(\frac{1}{p\theta}-\frac{1}{p_s^*}\right)^{-\frac{1-\gamma}{p_s^*-1+\gamma}}\left[\lambda \left(\frac{1}{1-\gamma}-\frac{1}{p\theta}\right)C(\Omega)\right] ^{\frac{p_s^*}{p_s^*-1+\gamma}}.\end{align*}
\end{lemma}
\begin{proof}
Let $\lambda>0$ and let $\{u_n\}\subset X_0$ be a sequence satisfying \eqref{ps 1}.
Without loss of generality, we shall assume that $\{u_n\}$ is such that $u_n$ is nonzero in $X_0$ for every $n\in\mathbb{N}$. 
However, due to the degenerate nature of Kirchhoff coefficient $\mathfrak{M}$, we distinguish two situations:
either \mbox{$\displaystyle\inf_{n\in\mathbb{N}}\|u_n\|=d>0$} or $\displaystyle\inf_{n\in\mathbb{N}}\|u_n\|=0$.\medskip

\noindent $\bullet$\,{\em Case} $\displaystyle\inf_{n\in\mathbb N}\|u_n\|=d>0$.
We first show that $\{u_n\}$ is bounded.
By $(\mathfrak{m}_{3})$, with $\tau=d^p$, there exists $\kappa=\kappa(d^p)>0$ such that
\begin{equation}\label{e2.2}
\mathfrak{M}(\|u_n\|^p)\geq \kappa\quad\mbox{for any }n\in\mathbb{N}.
\end{equation}
From this, together with $(\mathfrak{m}_{1})$ we get
\begin{align}\label{ps bdd}
\begin{split}
c+o(\| u_{n}\|)&=I_{\lambda}(u_{n})-\frac{1}{p_s^*}\langle I_{\lambda}^{'}(u_{n}),u_{n}\rangle\\
\geq &\frac{1}{p}\mathcal{M}(\|u_{n}\|^{p})-\frac{1}{p_s^*}\mathfrak{M}(\|u_{n}\|^{p})\|u_{n}\|^{p}
-\lambda\int_{\Omega}\left[G(x,u_n)-\frac{g(x,u_n)u_n}{p_s^*}\right]dx\\
\geq &\left(\frac{1}{p\theta}-\frac{1}{p_s^*}\right)\mathfrak{M}(\|u_{n}\|^{p})\|u_{n}\|^{p}-\lambda\left(\frac{1}{1-\gamma}-\frac{1}{p_s^*}\right)\int_{\Omega_1}(u_n^+)^{1-\gamma}\\
&-\lambda\left(\frac{1}{1-\gamma}-\frac{1}{p_s^*}\right)\int_{\Omega_2}\underline{u}^{1-\gamma}\\
\geq &\left(\frac{1}{p\theta}-\frac{1}{p_s^*}\right)\kappa\|u_{n}\|^{p}-\lambda\left(\frac{1}{1-\gamma}-\frac{1}{p_s^*}\right)C(\Omega)\|u_n\|^{1-\gamma}\\
&-\lambda\left(\frac{1}{1-\gamma}-\frac{1}{p_s^*}\right)\int_{\Omega_2}\underline{u}^{1-\gamma},
\end{split}
\end{align}
where
$$\Omega_1=\{x\in\Omega:u_n(x)>\underline{u}(x)\},\qquad\Omega_2=\{x\in\Omega:u_n(x)\leq\underline{u}(x)\}.$$
It is clear from the last inequality \eqref{ps bdd} that the sequence $\{u_{n}\}$ is bounded in $X_0$. Since the space $X_0$ is reflexive then there exists a subsequence of $\{u_{n}\}$ (still denoted by $\{u_{n}\}$) and $u\in X_0$ such that
\begin{equation}\label{convergence_result1}
\begin{array}{ll}
u_{n}\rightharpoonup u ~\text{ in}~ X_0,& \|u_n\|\to\mu\\
u_{n}\rightharpoonup u ~\text{ in}~ L^{p_s^*}(\Omega),&\|u_n-u\|_{p_s^*}\to l\\
u_{n}\rightarrow u ~\text{ in}~ L^{r}(\Omega) ~\text{for}~ 1\leq r<p_{s}^{*},&u_{n}(x)\rightarrow u(x) ~\text{a.e. in}~ \Omega
\end{array}
\end{equation}
as $n\to\infty$.
Furthermore, the sequence $\{\mathcal U_n\}$,
defined in $\mathbb R^{2N}\setminus\mbox{Diag\,}(\mathbb R^{2N})$ by
\begin{equation*}
(x,y)\mapsto\mathcal U_n(x,y)=\frac{|u_n(x)-u_n(y)|^{p-2}[u_n(x)-u_n(y)]}{|x-y|^{(N+ps)/p^{\prime}}},
\end{equation*}
is bounded in $L^{p^{\prime}}(\mathbb R^{2N})$ as well as $\mathcal U_n(x,y)\to\mathcal U(x,y)$ a.e. in $\mathbb R^{2N}$, where
\begin{equation*}
\mathcal U(x,y)
=\frac{|u(x)-u(y)|^{p-2}
[u(x)-u(y)]}{|x-y|^{(N+ps)/p^{\prime}}}.
\end{equation*}
Thus, going if necessary to a further subsequence, we get that $\mathcal U_n\rightharpoonup
\mathcal U$ in $L^{p^{\prime}}(\mathbb R^{2N})$ as $n\to\infty$ and so
\begin{equation}\label{x1}
\begin{aligned}
\lim_{n\to\infty}&\int_{Q}\frac{|u_n(x)-u_n(y)|^{p-2}
(u_n(x)-u_n(y))(\varphi(x)-\varphi(y))}{|x-y|^{N+ps}}dxdy\\
&=\int_{Q}\frac{|u(x)-u(y)|^{p-2}
(u(x)-u(y))(\varphi(x)-\varphi(y))}{|x-y|^{N+ps}}dxdy
\end{aligned}
\end{equation}
for any $\varphi\in X_0(\Omega)$, since
$$|\varphi(x)-\varphi(y)|\cdot|x-y|^{-(N+ps)/p}\in L^p(\mathbb R^{2N}).$$
While, by Vitali's convergence theorem we have
\begin{align}\label{conv sing comp}
\lim\limits_{n\rightarrow\infty}\int_{\Omega}g(x,u_n)u_n dx=\int_{\Omega}g(x,u)u dx.
\end{align}
It is not difficult to see that
\begin{align}\label{conv sing}
\lim\limits_{n\rightarrow\infty}\int_{\Omega}g(x,u_n)\varphi dx=\int_{\Omega}g(x,u)\varphi dx
\end{align}
for any $\varphi\in X_0$.

\vspace{0.2cm}
\noindent
{\bf{Proof of the claim:}} Let $\varphi\in X_0$. To prove the claim we first show that
\begin{align*}
\lim\limits_{n\rightarrow\infty}\int_{\Omega}\varphi g(x,u_n(x))dx=\int_{\Omega}\varphi g(x,u(x))dx.
\end{align*}
Let us define $S:=\{x:|u(x)|>\underline{u}(x)\}$, $T:=\{x:0<|u(x)|\leq\underline{u}(x)\}$. Since $u_n(x)\to u(x)$, hence for a sufficiently small $\epsilon>0$, $\underline{u}(x)<|u(x)|-\epsilon<|u_n(x)|$ for all $x$ except on a set of arbitrarily small measure, for all $n$. Thus we have 
\begin{align}\label{conv1}
\begin{split}
\lim\limits_{n\rightarrow\infty}\int_{\Omega}\varphi g(x,u_n(x))dx=&\lim\limits_{n\rightarrow\infty}\left(\int_{S}\varphi |u_n|^{-\gamma-1}u_ndx+\int_{T}\varphi \underline{u}^{-\gamma}dx\right)\\
=&\int_{S}\varphi |u|^{-\gamma-1}udx+\int_{T}\varphi \underline{u}^{-\gamma}dx\\
=& \int_{\Omega}g(x,u)\varphi dx.
\end{split}
\end{align}
\\\\
Consequently, by \eqref{ps 1} and \eqref{convergence_result1}-\eqref{conv sing} we get
\begin{align}\label{af1}
o(1)=&\langle I_{\lambda}'(u_n),u_n-u\rangle\nonumber\\
=&\mathfrak{M}(\|u_n\|^p)\int_{Q}\frac{|u_n(x)-u_n(y)|^{p-2}(u_n(x)-u_n(y))
[(u_n-u)(x)-(u_n-u)(y)]}{|x-y|^{N+ps}}dxdy\nonumber\\
&-\lambda\int_{\Omega}g(x,u_n)(u_n-u)dx-\int_{\Omega}|u_n|^{p_s^*-2}u_n(u_n-u)dx\\
=&\mathfrak{M}(\mu^p)(\mu^p-\|u\|^p)-\|u_n\|_{p_s^*}^{p_s^*}+\|u\|_{p_s^*}^{p_s^*}+o(1)\nonumber\\
=&\mathfrak{M}(\mu^p)(\|u_n-u\|^p)-\|u_n-u\|_{p_s^*}^{p_s^*}+o(1),\nonumber
\end{align}
as $n\to\infty$, where in the last step we used the classical Br\'ezis-Lieb lemma.
Thus we deduce the formula
\begin{align}\label{af2}
\mathfrak{M}(\mu^p)\underset{n\to\infty}\lim\|u_n-u\|^p=&\underset{n\to\infty}\lim\|u_n-u\|_{p_s^*}^{p_s^*}.
\end{align}
By \eqref{convergence_result1} and \eqref{af2} we obtain
\begin{align}\label{af3}
l^{p_s^*}\geq & S\mathfrak{M}(\mu^p)l^p.
\end{align}
When $l=0$ and since $\mu>0$ and $\mathfrak{M}$ admits a unique zero at $0$, hence $u_n\to u$ in $X_0$ thereby proving the result. 

If $l>0$, then by using \textcolor{red}{\eqref{af1}} we get
\begin{align}\label{af4}
\mathfrak{M}(\mu^p)(\mu^p-\|u_n\|^p)=& l^{p_s^*}.
\end{align}
Therefore we have
\begin{align}\label{af5}
(l^{p_s^*})^{ps/N}=&\mathfrak{M}(\mu^p)^{ps/N}(\mu^p-\|u\|^p)^{ps/N}\geq S\mathfrak{M}(\mu^p).
\end{align}
Since the exact behaviour of $\mathfrak{M}$ is unknown, hence we must consider two situations.

\medskip

\noindent $\bullet$\,{\em Subcase}  $\mu\in(0,1)$.~ By \eqref{af5} and $(\mathfrak{m_2})$ we get
\begin{align}\label{af6}
\mu^{p^2s/N}\geq(\mu^p-\|u\|^p)^{ps/N}\geq S\mathfrak{M}(\mu^p)^{p/p_s^*}\geq S\mathfrak{m}_0^{p/p_s^*}\mu^{\frac{p^2(\theta-1)}{p_s^*}}.
\end{align}
Furthermore, since $N<ps\theta/(\theta-1)=ps\theta'$, we have
\begin{align}\label{af7}
\mu^p&\geq \left(\textcolor{red}{S}\mathfrak{m}_0^{p/p_s^*}\right)^{\frac{N}{ps\theta-N(\theta-1)}}.
\end{align}
Of course the restriction $N/(p\theta')<s$ follows from the fact $1<\theta<\frac{p_s^*}{p}$. By using $(\mathfrak{m}_2)$, \eqref{af5}, \eqref{af7} we get
\begin{align}\label{af8}
l^{p_s^*}\geq(S\mathfrak{M}(\mu^p))^{N/ps}\geq(S\mathfrak{m}_0\mu^{p(\theta-1)})^{N/ps}\geq(\mathfrak{m}_0^{1/\theta}S)^{\frac{N\theta}{ps\theta-N(\theta-1)}}.
\end{align}
Arguing similarly to \eqref{ps bdd}, by$(\mathfrak{m}_{1})$ we get 
\begin{align}\label{af9}
I_{\lambda}(u_n)-\frac{1}{p\theta}\langle I_{\lambda}'(u_n),u_n\rangle\geq &\frac{1}{p}\mathcal{M}(\|u_n\|^p)-\frac{1}{p\theta}\mathfrak{M}(\|u_n\|^p)\|u_n\|^p\nonumber\\
&-\lambda\left(\frac{1}{1-\gamma}-\frac{1}{p\theta}\right)\int_{\Omega_1}|u_n|^{1-\gamma}dx+\left(\frac{1}{p\theta}-\frac{1}{p_s^*}\right)\|u_n\|_{p_s^*}^{p_s^*}\\
&-\lambda\left(\frac{1}{1-\gamma}-\frac{1}{p\theta}\right)\int_{\Omega_2}\underline{u}^{1-\gamma}\nonumber\\
\geq & \left(\frac{1}{p\theta}-\frac{1}{p_s^*}\right)\|u_n\|_{p_s^*}^{p_s^*}-\lambda\left(\frac{1}{1-\gamma}-\frac{1}{p\theta}\right)\int_{\Omega_1}|u_n|^{1-\gamma}dx\nonumber\\
&-\lambda\left(\frac{1}{1-\gamma}-\frac{1}{p\theta}\right)\int_{\Omega_2}\underline{u}^{1-\gamma}.\nonumber
\end{align}
On passing the limit $n\to\infty$ and by using \eqref{ps 1}, \eqref{convergence_result1}, the Br\'ezis-Lieb lemma, H\"{o}lder's inequality and Young's inequality (to the second term) we obtain
\begin{align}\label{af10}
\begin{split}
c\geq &\left(\frac{1}{p\theta}-\frac{1}{p_s^*}\right)(l^{p_s^*}+\|u\|_{p_s^*}^{p_s^*})-\lambda\left(\frac{1}{1-\gamma}-\frac{1}{p\theta}\right)\int_{\Omega_1}(u^+)^{1-\gamma}dx\\
&-\lambda\left(\frac{1}{1-\gamma}-\frac{1}{p\theta}\right)\int_{\Omega}\underline{u}^{1-\gamma}\\
\geq & \left(\frac{1}{p\theta}-\frac{1}{p_s^*}\right)(l^{p_s^*}+\|u\|_{p_s^*}^{p_s^*})-\lambda C(\Omega)\left(\frac{1}{1-\gamma}-\frac{1}{p\theta}\right)\|u\|_{p_s^*}^{1-\gamma}dx\\
&-\lambda\left(\frac{1}{1-\gamma}-\frac{1}{p\theta}\right)\int_{\Omega}\underline{u}^{1-\gamma}\\
\geq &\left(\frac{1}{p\theta}-\frac{1}{p_s^*}\right)(l^{p_s^*}+\|u\|_{p_s^*}^{p_s^*})-\left(\frac{1}{p\theta}-\frac{1}{p_s^*}\right)\|u\|_{p_s^*}^{p_s^*}\\
&-\left(\frac{1}{p\theta}-\frac{1}{p_s^*}\right)^{-\frac{1-\gamma}{p_s^*-1+\gamma}}\left[\lambda \left(\frac{1}{1-\gamma}-\frac{1}{p\theta}\right)C(\Omega)\right] ^{\frac{p_s^*}{p_s^*-1+\gamma}}\\
&-\lambda\left(\frac{1}{1-\gamma}-\frac{1}{p\theta}\right)\int_{\Omega}\underline{u}^{1-\gamma}.
\end{split}
\end{align}
Finally by \eqref{af8} we obtain 
\begin{align*}
c\geq \left(\frac{1}{p\theta}-\frac{1}{p_s^*}\right)(\mathfrak{m}_0^{1/\theta}S)^{p_s^*/(p_s^*-p\theta)}
&-\left(\frac{1}{p\theta}-\frac{1}{p_s^*}\right)^{-\frac{1-\gamma}{p_s^*-1+\gamma}}\left[\lambda \left(\frac{1}{1-\gamma}-\frac{1}{p\theta}\right)C(\Omega)\right] ^{\frac{p_s^*}{p_s^*-1+\gamma}}\\
&-\lambda\left(\frac{1}{1-\gamma}-\frac{1}{p\theta}\right)\int_{\Omega}\underline{u}^{1-\gamma}
\end{align*}
which contradicts $c<c_*$ and \eqref{cond1}. Hence $l=0$.

\medskip

\noindent $\bullet$\,{\em Subcase}  $\mu>1$.~ By using \eqref{af5} and $(\mathfrak{m}_3)$ with $\tau=1$, we get
 \begin{align}\label{af12}
 l^{p_s^*}\geq (\kappa S)^{N/ps},
 \end{align}
with $\kappa=\kappa(1)>0$. Therefore, by \eqref{af10} we obtain
\begin{align*}
c\geq \left(\frac{1}{p\theta}-\frac{1}{p_s^*}\right)(\kappa S)^{p_s^*/(p_s^*-p)}
&-\left(\frac{1}{p\theta}-\frac{1}{p_s^*}\right)^{-\frac{1-\gamma}{p_s^*-1+\gamma}}\left[\lambda \left(\frac{1}{1-\gamma}-\frac{1}{p\theta}\right)C(\Omega)\right] ^{\frac{p_s^*}{p_s^*-1+\gamma}}\\
&-\lambda\left(\frac{1}{1-\gamma}-\frac{1}{p\theta}\right)\int_{\Omega}\underline{u}^{1-\gamma}.
\end{align*}
This again leads to a contradiction. Hence the proof.\medskip

\noindent $\bullet$\,{\em Case} $\displaystyle\inf_{n\in\mathbb N}\|u_n\|=0$.~ This means, either $0$ is an accumulation point for real sequence $\{\|u_n\|\}$ and hence there exists a subsequence of $\{u_n\}$ such that it strongly goes to $0$. 

This concludes the proof of the lemma.

\end{proof}

\section{Auxiliary Results and Proof of Main Theorem. }
\noindent In this section, we will establish the existence of arbitrarily many solutions to the problem \eqref{cut p}. Prior to that let us define some useful tools to be used to guarantee the existence of solutions.
\begin{definition}[\bf{Genus}\cite{Rabinowitz1986}]\label{genus}
	Let $X$ be a Banach space and $A\subset X$. A set $A$ is said to be symmetric if $u\in A$ implies $(-u)\in A$. Let $A$ be a close, symmetric subset of $X$ such that $0\notin A$. We define a genus $\sigma(A)$ of $A$ by the smallest integer $k$ such that there exists an odd continuous mapping from $A$ to $\mathbb{R}^{k}\setminus\{0\}$. We define $\sigma(A)=\infty$, if no such $k$ exists.
\end{definition}
\noindent The next Proposition is due to \cite{Rabinowitz1986} pertaining to some properties of genus.
Let $\Gamma$ denotes the family of all closed subsets of $X\setminus\{0\}$ which are symmetric with respect to the origin.
\begin{lemma}\label{lemma genus}
	Let $A, B\in\Gamma$. Then
	\begin{enumerate}
		\item $A\subset B\Rightarrow\sigma(A)\leq\sigma(B)$.
		\item Suppose  $A$ and $B$ are homeomorphic via an odd map, then $\sigma(A)=\sigma(B)$.
		\item $\sigma(\mathbb S^{N-1})=N$, where $\mathbb S^{N-1}$ is the sphere in $\mathbb{R}^{N}$.
		\item $\sigma(A\cup B)\leq \sigma(A)+\sigma(B)$.
		\item $\sigma(A)<\infty\Rightarrow\sigma(A\setminus B)\geq\sigma(A)-\sigma(B)$.		
		\item For every compact subset $A$ of $X$, $\sigma(A)<\infty$ and there exists $\delta>0$ such that $\sigma(A)=\sigma(N_{\delta}(A))$ where $N_{\delta}(A)=\{x\in X:d(x,A)\leq\delta\}.$		
		\item Suppose $Y\subset X$ is a subspace of $X$ such that $codim(Y)=k$ and $\sigma(A)>k$, then $A\cap X_{0}\neq\emptyset.$
	\end{enumerate}	
\end{lemma}
\noindent We will use the following version of the symmetric Mountain Pass Theorem due to Kajikiya \cite{Kajikiya2005}.
\begin{theorem}\label{sym mountain}
	Let $X$ be an infinite dimensional Banach space and $I\in C^1(X,\mathbb{R})$ satisfies the following
	\begin{itemize}
		\item[(i)] $I$ is even, bounded below, $I(0)=0$ and $I$ satisfies the $(PS)_c$-condition for $c>0$.
		\item[(ii)] For each $n\in\mathbb{N}$, there exists an $A_n\in\Gamma_n$ such that $\sup\limits_{u\in A_n}I(u)<0.$ 
	\end{itemize}
	Then either $(1)$ or $(2)$ below holds.
	\begin{itemize}
		\item[(1)] There exists a sequence $\{u_n\}$ such that $I'(u_n)=0$, $I(u_n)<0$ and $u_n\longrightarrow0$ in $X$.
		\item[(2)] There exist two sequences $\{u_n\}$ and $\{v_n\}$ such that $I'(u_n)=0$, $I(u_n)=0$, $u_n\neq0$, $\lim\limits_{n\rightarrow\infty}u_n=0$; $I'(v_n)=0$, $I(v_n)<0$, $\lim\limits_{n\rightarrow\infty}u_n=0$ and $\{v_n\}$ converges to a non-zero limit.
	\end{itemize}
\end{theorem}
\begin{remark}
	It is important here to mention that in both the cases from Theorem \ref{sym mountain}, we obtain a sequence $\{u_n\}$ of critical points such that $I'(u_n)=0$, $I(u_n)=0$, $u_n\neq0$, $\lim\limits_{n\rightarrow\infty}u_n=0$.
\end{remark}
\noindent Observe that $\lim\limits_{t\rightarrow\infty}I_{\lambda}(tu)=-\infty$. Therefore the functional $I_{\lambda}$ is not bounded from below. Hence, we will use a technique from \cite{Garcia1991} to overcome this difficulty.

Let us define $\beta_\lambda: [0,\infty)\to\mathbb R$ by
\begin{align}\label{g1}
\beta_{\lambda}(t)=\frac{\mathcal{M}(1)}{p}t^{p\theta}-\lambda \frac{C_{1-\gamma}}{1-\gamma}t^{1-\gamma}-\frac{1}{p_s^*S^{p_s^*/p}}t^{p_s^*},
\end{align}
where $C_{1-\gamma}>0$ is given by Lemma \ref{embedding} and $S$ is as in \eqref{sobolev const}.
Since $1-\gamma<p\theta$ we see that $\beta_{\lambda}(t)<0$ for $t$ near zero and due to $1-\gamma<p\theta<p^*$ there exists $\lambda_1>0$ such that $\beta_{\lambda}$ attains its positive maximum for any $\lambda \in(0,\lambda_1)$. Let $r_0(\lambda)$ and  $r_1(\lambda)$ be the unique roots of $\beta_{\lambda}$ such that $0<r_0(\lambda)<r_1(\lambda)$. Now, we claim that
\begin{equation}\label{claimr}
R_0(\lambda)\to 0\quad\mbox{as }\lambda\to 0.
\end{equation}

\vspace{0.2cm}
\noindent {\bf Proof of the claim:} From $\beta_{\lambda}(r_0(\lambda))=0$ and $\beta_{\lambda}'(r_0(\lambda))>0$ we have
\begin{align}\label{root-1}
	\frac{\mathcal{M}(1)}{p}r_0(\lambda)^{p\theta}=\lambda \frac{C_{1-\gamma}}{1-\gamma}r_0(\lambda)^{1-\gamma}+\frac{1}{p_s^*S^{p_s^*/p}}r_0(\lambda)^{p_s^*}
\end{align}
and
\begin{align}\label{root-2}
\mathcal{M}(1)	r_0(\lambda)^{p\theta-1}>\lambda C_{1-\gamma} r_0(\lambda)^{-\gamma}+\frac{1}{S^{p_s^*/p}}r_0(\lambda)^{p_s^*-1}
\end{align}
for any $\lambda \in (0,\lambda_1)$. From \eqref{root-1} we know that $r_0(\lambda)$ is bounded since $p\theta<p_s^*$. Suppose that $r_0(\lambda)\to R>0$ as $\lambda \to 0$. Then we obtain from \eqref{root-1} and \eqref{root-2}
\begin{align*}
	\frac{\mathcal{M}(1)}{p}R^{p\theta}= \frac{1}{p_s^*S^{p_s^*/p}}R^{p_s^*}
	\quad\text{and}\quad 
	\mathcal{M}(1)R^{p\theta-1}\geq \frac{1}{S^{p_s^*/p}}R^{p_s^*-1},
\end{align*}
which is a contradiction since $p\theta<p_s^*$. This proves the claim.
\\\\
From \eqref{claimr}, there exists $\lambda_2>0$ such that $r_0(\lambda)<1$ for any $\lambda\in(0,\lambda_2)$. This implies that $r_0(\lambda)<\min\{r_1(\lambda),1\}$. Thus, for any $\lambda\in(0,\min\{\lambda_1,\lambda_2\})$ we choose a $C^\infty$-function $\nu:[0,\infty)\to [0,1]$ such that
\begin{align}\label{def-tau}
	\nu(t):=
	\begin{cases}
		1 & \text{if } t \in [0,r_0(\lambda)],\\
		0 & \text{if }t \in [\min\{r_1(\lambda),1\},\infty).
	\end{cases}
\end{align}
Then, we can introduce the truncated energy functional as follows.
\begin{align}\label{main 3}
\overline{I}_{\lambda}(u)=\frac{1}{p}\mathcal{M}(\| u\|^{p})\|u\|^p-\lambda\int_{\Omega}G(x,u)dx-\frac{1}{p_s^*}\nu(\|u\|)\int_{\Omega}|u|^{p_s^*}dx.
\end{align}
It is clear that $\overline{I}_{\lambda}$ is coercive and bounded from below. Again from \cite[Lemma 6.4]{Saoudi2019}, we conclude that $\overline{I}_{\lambda} \in C^1(X_0,\mathbb R)$. Also, note that if $\|u\|\leq r_0(\lambda)< \min\{r_1(\lambda),1\}$, then $\overline{I}_{\lambda}=I_\lambda(u)$.

%Now define
%\begin{equation}\label{defn h bar}
%\overline{h}(t)=\frac{\mathfrak{m}_{0}}{p\theta}t^{p}-\lambda\frac{S^{(\gamma-1)/p}}{1-\gamma}t^{1-\gamma}-\frac{{S}^{-p_{s}^{*}/p}}{p_{s}^*}t^{p_{s}^{*}}\nu(t)
%\end{equation}
%From \eqref{main 1}, it is clear that
%\begin{align}\label{main 4}
%\overline{I}_{\lambda}(u)\geq\overline{h}(\|u\|).
%\end{align}
%One can easily conclude that $\overline{h}(x)\geq h(x)$ whenever $x\geq0$, $\overline{h}(x)=h(x)$ for $0\leq x\leq r_{0}$, $\overline{h}(x)\geq 0$, for $r_{0}<x\leq r_{1}$ and if $x>r_{1}$, then $\overline{h}(x)>0$ since the function $\overline{h}(x)= x^{1-\gamma}((\mathfrak{m}_{0}/p\theta)x^{p-1+\gamma}-\lambda(S^{(\gamma-1)/p}/(1-\gamma))$ is strictly increasing. Therefore, $\overline{h}(x)\geq 0$ for all $x\geq r_{0}.$\\

We now prove the following auxiliary Lemma for the truncated functional $\overline{I}_{\lambda}$ to apply the symmetric mountain pass theorem.
\begin{lemma}\label{lemma smpt}
	There exists $\lambda_*>0$ such that for all $\lambda\in(0,\lambda_*)$, we have
	\begin{enumerate}[label=(\roman*)]
		\item $\|u\|<r_{0}(\lambda)$ whenever $\overline{I}_{\lambda}(u)<0$. In addition, $\overline{I}_{\lambda}(u)=I_{\lambda}(u)$ for all $u\in N_{\delta}(u)$.
		\item For every $c<0$, the functional $\overline{I}_{\lambda}$ satisfies a local $(PS)_c$-condition.
	\end{enumerate}
\end{lemma}
\begin{proof}
Let us set $\lambda_*\leq\min\{\lambda_1,\lambda_2,\lambda_3,\lambda_4\}$ where $\lambda_1$ and $\lambda_2$ are chosen for the definition of $\nu$ in \eqref{def-tau}, while $\lambda_3$ is sufficiently small so that
\begin{align}\label{cond2}
\min&\left\{\left(\frac{1}{p\theta}-\frac{1}{p_s^*}\right)(\kappa S)^{p_s^*/(p_s^*-p)},\left(\frac{1}{p\theta}-\frac{1}{p_s^*}\right)(m_0^{1/\theta}S)^{\frac{p_s^*\theta}{p_s^*-p\theta}}\right\}\nonumber\\
&-\left(\frac{1}{p\theta}-\frac{1}{p_s^*}\right)^{-\frac{1-\gamma}{p_s^*-1+\gamma}}\left[\lambda_3 \left(\frac{1}{1-\gamma}-\frac{1}{p\theta}\right)C(\Omega)\right] ^{\frac{p_s^*}{p_s^*-1+\gamma}}\\
&-\lambda\left(\frac{1}{1-\gamma}-\frac{1}{p\theta}\right)\int_{\Omega}\underline{u}^{1-\gamma}>0,\nonumber
\end{align}	
while $\lambda_4=\kappa/C_{1-\gamma}\,\theta$.
Let $\lambda\in(0,\lambda_*)$.

\vspace{0.2cm}
	\begin{enumerate}[label=(\roman*)]
		\item Suppose that $\overline{I}_{\lambda}(u)<0$. We distinguish two different cases.
		\medskip

\noindent $\bullet$\,{\em Case} $\|u\|\geq 1$.~ This case can not occur. Indeed, by the definition of $\nu$ in \eqref{def-tau} and by $(\mathfrak{m}_{1})$, $(\mathfrak{m}_{3})$ with $\tau=1$, we get that
	\begin{align}\label{estimate-below10}
			\overline{I}_{\lambda}(u)\geq \frac{\kappa}{p\theta}\|u\|^p -\frac{\lambda}{1-\gamma} C_{1-\gamma} \|u\|^{1-\gamma}
			=\phi_\lambda (\|u\|),
	\end{align}
	where $\phi_\lambda: [1,\infty)\to\mathbb R$ is given by
	\begin{align*}
		\phi_\lambda(t):=\frac{\kappa}{p\theta}t^p-\frac{\lambda}{1-\gamma} C_{1-\gamma} t^{1-\gamma}.
	\end{align*}
	It is clear that $\phi_\lambda$ has a global minimum point at 
	\begin{align*}
		t_0=\left(\lambda C_{1-\gamma}\frac{\theta}{\kappa}\right)^{\frac{1}{p-1+\gamma}}
	\end{align*}
	with
	\begin{align*}
		\phi_\lambda(t_0)=\frac{\kappa}{p\theta}\left(\lambda C_{1-\gamma}\frac{\theta}{\kappa}\right)^{\frac{p}{p-1+\gamma}} \left(1-\frac{p}{1-\gamma}\right)<0,
	\end{align*}
	since $1-\gamma<p$. We point out that $\phi_\lambda(t)\geq 0$ if and only if $t \geq t_0$. Hence, being $\lambda< \lambda_4=\kappa/C_{1-\gamma}\theta$ we have $\displaystyle\min_{t\in[1,\infty]}\phi_\lambda(t)\geq 0$ which yields, joint with \eqref{estimate-below10}, that $\overline{I}_{\lambda}(u) \geq 0$ for any $\|u\|\geq 1$. This gives the desired contradiction.
			\medskip

\noindent $\bullet$\,{\em Case} $\|u\|< 1$.~ By integrating $(\mathfrak{m}_{1})$, we get
$$
\mathcal{M}(t)\geq\mathcal{M}(1)t^\theta,\quad\mbox{for any }t\in[0,1].
$$
From this, we get
	\begin{align*}%\label{estimate-below1}
		\begin{split}
			\overline{I}_{\lambda}(u)
			&\geq \frac{\mathcal{M}(1)}{p}\|u\|^{p\theta} -\frac{\lambda}{1-\gamma} C_{1-\gamma} \|u\|^{1-\gamma}
			-\frac{1}{p_s^*S^{p_s^*/p}}\|u\|^{p_s^*}\nu(\|u\|)\\
			&=\widehat{\beta}_\lambda (\|u\|),
		\end{split}
	\end{align*}
	where
	\begin{align*}
		\widehat{\beta}_\lambda(t):=\frac{\mathcal{M}(1)}{p}t^{p\theta}-\frac{\lambda}{1-\gamma} C_{1-\gamma} t^{1-\gamma}-\frac{1}{p_s^*S^{p_s^*/p}}t^{p_s^*}\nu(t).
	\end{align*}
	Since $0\leq\nu\leq1$, we note that 
	\begin{align}\label{beta-1-hat}
	\widehat{\beta}_\lambda(t)\geq\beta_\lambda(t)\geq0\quad\mbox{for any }t \in[r_0(\lambda),\min\{r_1(\lambda),1\}],
	\end{align}
	where last inequality follows by the construction of the roots $r_0(\lambda)$ and $r_1(\lambda)$ for $\beta_\lambda$. 
	
	\noindent
	Hence, if $\min\{r_1(\lambda),1\}=1$, then from $\overline{I}_{\lambda}(u) < 0$ and \eqref{beta-1-hat} we obtain that $\|u\|<r_0(\lambda)$. 
	
	\noindent
	While, if $\min\{r_1(\lambda),1\}=r_1(\lambda)$, considering  $r_1(\lambda)<\|u\|<1$ and arguing similarly to \eqref{estimate-below10} we get
	\begin{align*}
	\overline{I}_{\lambda}(u)\geq\widehat{\phi}_\lambda (\|u\|)\qquad\mbox{ with }\ \ 
	\widehat{\phi}_\lambda(t):=\frac{\mathcal{M}(1)}{p}t^{p\theta}-\frac{\lambda}{1-\gamma} C_{1-\gamma} t^{1-\gamma},
	\end{align*}
	from which we can proceed exactly as in the previous case to get a contradiction. Considering  $r_0(\lambda)<\|u\|\leq r_1(\lambda)$, from $\overline{I}_{\lambda}(u) < 0$ and \eqref{beta-1-hat} we get another contradiction. Hence, we obtain again $\|u\|<r_0(\lambda)$, completing the first part of (i). 
		
\noindent
Moreover, for any $u$ in a neighbourhood say $N_{\delta}(u)=B\left(0,\frac{r_0(\lambda)}{2}\right)$ we have that $\overline{I}_{\lambda}(u)=I_{\lambda}(u)$.
%		Then from \eqref{main 4} and the succeeding calculations, we can conclude that $\overline{I}_{\lambda}(u)\geq\overline{h}(u)\geq0$. This gives a contradiction to our assumption. Therefore, $\|u\|<r_0$. In addition, by using the fact that $\overline{h}(x)=h(x)$ for $0\leq x\leq r_{0}$, one may guarantee that there exists a $\delta>0$, small enough such that $\overline{I}_{\lambda}(v)=I_{\lambda}(v)$ for all $v\in N_{\delta}(u)$.
\vspace{0.2cm}
		\item Note that any Palais-Smale sequence for $\overline{I}_{\lambda}$ has to be bounded since $\overline{I}_{\lambda}$ is coercive. Therefore, since $\lambda<\lambda_3$, by Lemma \ref{ps limit} with \eqref{cond1} and \eqref{cond2}, we observe a local Palais-Smale condition for $I_{\lambda}\equiv \overline{I}_{\lambda}$ at any level $c<0$. 
%		and $\{u_{n}\}\subset X_0$ be a $(PS)_{c}$ sequence for the functional $\overline{I}_{\lambda}$. Then we have $\overline{I}_{\lambda}(u_{n})<0$, $\overline{I}_{\lambda}^{'}(u_{n})\rightarrow 0$. From the conclusion in $(ii)$, using the coerciveness of $\overline{I}_{\lambda}$ we get the sequence $\{u_{n}\}$ bounded in $X_0$. Thus from $(iii)$ we get $\|u_{n}\|<r_{0}$ and hence $\overline{I}_{\lambda}(u_{n})=I_{\lambda}(u_{n})$ and $\overline{I}_{\lambda}^{'}(u_{n})=I_{\lambda}^{'}(u_{n})$. Finally from Lemma \ref{ps limit}, we can obtain that there exists $\lambda_{0}>0$ such that for all $0<\lambda<\lambda_{0}$, the functional $\overline{I}_{\lambda}$ satisfies $(PS)_{c}$-condition.		
	\end{enumerate}
\end{proof}
\noindent We now prove the following technical property of $\overline{I}_{\lambda}$ which guarantees the existence of of a subset of $X_0$ of genus at least $n$ for every $n\in\mathbb{N}$.
\begin{lemma}\label{lemma genus n}
	For every $\lambda>0$ and $k\in\mathbb{N}$ there exists $\varepsilon=\varepsilon(\lambda,k)<0$ such that
	$$\sigma(\{u\in X_0:\,\, \overline{I}_{\lambda}(u)\leq\varepsilon\})\geq k.$$
\end{lemma}
\begin{proof}
	Let $k\in \mathbb{N}$ and $\lambda>0$ be given. Let $Y^k$ be a $k$-dimensional subspace of $X_0$. Now for each $u\in Y^k$ such that $0<\|u\|\leq r_0(\lambda)$ we have
	\begin{align}\label{main 9}
	\overline{I}_{\lambda}(u)&\leq\frac{1}{p}\mathcal{M}(\|u\|^{p})\|u\|^{p}-\frac{\lambda}{1-\gamma}\int_{\Omega}G(x,u)dx-\frac{1}{p_{s}^*}\int_{\Omega}|u|^{p_{s}^{*}}dx\nonumber\\
	&\leq\frac{M}{p}\|u\|^p-\frac{\lambda}{1-\gamma}\int_{\Omega}G(x,u)dx.
	\end{align}
	where $M=\max\limits_{0\leq t\leq r_{0}(\lambda)}\mathcal{M}(t)<\infty$. Since $0<1-\gamma<1$, we loose the advantages of having concave-convex nonlinearity. More precisely, we cannot use an immediate equivalence of norms over finite dimensional space to conclude $\overline{I}_{\lambda}(u)<0$. We instead use the following bypass route to overcome this difficulty. Note that, using the fact that $\dim(Y^k)=k$, we can say there exist $c>0$ {such that $$c^{-1}\|u\|_1\leq\|u\|\leq{c}\|u\|_1.$$}
	Therefore, whenever $\|u\|\leq r_0(\lambda)$, we get $\|u\|_{1}\leq cr_0(\lambda)$. 
From the truncation function $G$ we have the following:
%$$  
%\int_{\Omega}{G}(x,u) \geq 
%	\int_{\Omega}(\underline{u}(x))^{1-\gamma}dx.$$

\[   
\int_{\Omega}{G}(x,u) \geq 
\begin{cases}
\int_{\Omega}\frac{1}{1-\gamma}(\underline{u}(x))^{1-\gamma}dx, &~\text{if}~|u(x)|>\underline{u}(x)\\
(c')^{-\gamma}\int_{\Omega}|u|dx,&~\text{if}~ |u(x)|\leq \underline{u}(x),
\end{cases}\] 	
where $c'$ is a bound of $\underline{u}$ in $\Omega$.
	Thus, we get
	\begin{align}\label{main 11}
	\begin{split}
	\overline{I}_{\lambda}(u)&\leq\frac{1}{p}\mathcal{M}(\|u\|^{p})\|u\|^{p}-\lambda\int_{\Omega}G(x,u)dx\\
	&\leq \frac{1}{p}\mathcal{M}(\|u\|^{p})\|u\|^{p}-\lambda
\begin{cases}
\int_{\Omega}\frac{1}{1-\gamma}(\underline{u}(x))^{1-\gamma}dx, &~\text{if}~|u(x)|>\underline{u}(x)\\
(c')^{-\gamma}\int_{\Omega}|u|dx,&~\text{if}~ |u(x)|\leq \underline{u}(x).
\end{cases}
\end{split}
	\end{align}
	Furthermore,
	
	\begin{align}\label{eq11}
	\overline{I}_{\lambda}(u)&\leq \frac{1}{p}\mathcal{M}(\|u\|^{p})\|u\|^{p}-\lambda
\begin{cases}
\int_{\Omega}\frac{1}{1-\gamma}(\underline{u}(x))^{1-\gamma}dx, &~\text{if}~|u(x)|>\underline{u}(x)\\
(c')^{-\gamma}c^{-1}\|u\|,&~\text{if}~ |u(x)|\leq \underline{u}(x).
\end{cases}
	\end{align}
	Now choose, $\rho>0$ and $R>0$ such that
	$$0<\rho<R<\min\left\{r_0(\lambda),\left[\frac{p\lambda\int_{\Omega}\underline{u}^{1-\gamma}dx}{M(1-\gamma)}\right]^{\frac{1}{p}},\left[\frac{p\lambda\|u\|}{M (c')^{\gamma}c}\right]^{\frac{1}{p-1}}\right\}$$ and consider the subset of $Y^k$, defined as $A_k=\{u\in Y^k:\,\,\|u\|=\rho\}$.	Thus for any $u\in A_k$, we obtain
	\begin{align}\label{main 12}
	\overline{I}_{\lambda}(u)&\leq\rho^{1-\gamma}\left(\frac{M}{p}\rho^{p-1+\gamma}-\frac{\lambda\rho^{-1+\gamma}\int_{\Omega}\underline{u}^{1-\gamma}dx}{1-\gamma}\right)\nonumber\\
	&\leq R^{1-\gamma}\left(\frac{M}{p}R^{p-1+\gamma}-\frac{\lambda R^{-1+\gamma}\int_{\Omega}\underline{u}^{1-\gamma}dx}{1-\gamma}\right)\nonumber\\
	&<0.
	\end{align}
	This implies that for every $u\in A_k$ there exists $\varepsilon<0$ such that
	$$\overline{I}_{\lambda}(u)\leq\varepsilon<0.$$
	Therefore, $A_k\subset\{u\in X_0:\,\, \overline{I}_{\lambda}(u)\leq\varepsilon\}.$ Finally by Lemma \ref{lemma genus}, one can conclude that
	$$\sigma(\{u\in X_0:\,\, \overline{I}_{\lambda}(u)\leq\varepsilon\})\geq\sigma(A_k\cap Y^k)\geq k.$$
	This completes the proof.
\end{proof}
\noindent We now define the following notations. For each $n\in\mathbb{N}$, let us define the following $$\Gamma_n=\{A_n\subset X:\,\, A_n~\text{is closed, symmetric and}~ 0\notin A_n~\text{such that}~ \sigma(A_n)\geq n\},$$ 
$$K_{c}=\{u\in X_0:\,\,\overline{I}'_{\lambda}(u)=0,\,\overline{I}_{\lambda}(u)=c\}~\text{and}~c_{n}=\inf_{A\in\Gamma_{n}}\sup_{u\in A}\overline{I}_{\lambda}(u).$$
\begin{lemma}\label{lemma cn}
	For $\lambda\in(0,\lambda_*)$ and each $n\in\mathbb{N}$, the energy $c_n$ is a critical value of $\overline{I}_{\lambda}$. Moreover, $c_n<0$ and $\lim\limits_{n\rightarrow\infty}c_n=0.$
\end{lemma}
\begin{proof}
	Let $\lambda\in(0,\lambda_*)$ as in Lemma \ref{lemma smpt}.
	By using Lemma \ref{lemma smpt}(ii) and Lemma \ref{lemma genus n}, we conclude that
	\begin{equation}\label{main pf 1}
		-\infty<c_n<0.
	\end{equation}
	Again, for all $n\in\mathbb{N}$, $c_n\leq c_{n+1}$, since $\Gamma_{n+1}\subset\Gamma_n$. Therefore from \eqref{main pf 1}, we get $\lim_{n\rightarrow\infty}c_n=c\leq0.$ Proceeding similar to \cite[Proposition 9.33]{Rabinowitz1986}, we will show that $c=0$. We will prove it by contradiction. Observe that for $c<0$, the functional $\overline{I}_{\lambda}$ satisfies $(PS)_c$-condition. Therefore the set $K_c$ is compact. Then from the Lemma \ref{lemma genus}, there exists a $\delta>0$ such that $\sigma(N_{\delta}(K_c))=\sigma(K_c)=m<\infty$. Again, since $c<0$, then by using the Deformation Lemma \cite{Willem1997}, there exists $\varepsilon>0$ with $\varepsilon+c<0$ and an odd homeomorphism $\eta:X_0\rightarrow X_0$ such that 
	\begin{equation}\label{deform}
		\eta(A^{c+\varepsilon}\setminus N_{\delta}(K_c))\subset A^{c-\varepsilon},
	\end{equation}
	where $A^{c}=\{u\in X_0:\,\,\overline{I}_{\lambda}(u)\leq c\}$. Now, the sequence $\{c_n\}$ is monotonically increasing and $\lim_{n\rightarrow\infty}c_n=c$. Thus there exists $n\in\mathbb{N}$ such that $c_n>c-\varepsilon$ and $c_{n+m}<c$. We now choose, $A\subset\Gamma_{n+m}$ such that $\sup_{u\in A}\overline{I}_{\lambda}(u)<c+\varepsilon$. In other words, $A\subset A^{c+\varepsilon}$. Then from Lemma \ref{lemma genus}, we have
	$$\sigma(\overline{A\setminus N_{\delta}(K_c)})\geq\sigma(A)- \sigma(N_{\delta}(K_c))\geq n~\text{and}~\sigma(\eta(\overline{A\setminus N_{\delta}(K_c)}))\geq n.$$
	This implies that $\eta(\overline{A\setminus N_{\delta}(K_c)})\in\Gamma_n$. Hence, we obtain $$\sup\limits_{u\in\eta(\overline{A\setminus N_{\delta}(K_c)})}\overline{I}_{\lambda}(u)\geq c_n>c-\varepsilon.$$
	This gives a contradiction and hence we conclude $c_n\rightarrow0$. Finally from \cite{Rabinowitz1986}, it is easy to prove that for each $n\in\mathbb{N}$, the energy $c_n$ is a critical value of $\overline{I}_{\lambda}$ This completes the proof.
\end{proof}
\begin{proof}[{\it Proof of Theorem \ref{thm main}:}]
	\noindent Observe that from Lemma \ref{lemma smpt}, we have $\overline{I}_{\lambda}(u)=I_{\lambda}(u)$ whenever $\overline{I}_{\lambda}(u)<0$. We now have all the ingredients required for the symmetric mountain pass theorem. Therefore, from Lemma \ref{lemma genus}, Lemma \ref{lemma smpt}, Lemma \ref{lemma genus n} and Lemma \ref{lemma cn} one can easily verify that all the hypotheses of Theorem \ref{sym mountain} are satisfied. Therefore, we conclude that there exists infinitely many critical points of the functional $I_{\lambda}$. However, since $u_n\rightarrow 0$ in $X_0$ as $n\rightarrow\infty$, therefore there is a possibility of $u_n$ becoming smaller than $\underline{u}$ after a finite number of terms. Hence, on suitably choosing $\lambda$ to be small, one can obtain $k$ solutions to the problem \eqref{main p}, for arbitrarily large $k$. Moreover, since $\overline{I}_{\lambda}(u)=\overline{I}_{\lambda}(|u|)$, we can choose the solutions to be nonnegative.
%	We will conclude our main Theorem by showing the non-negativity of weak solutions as follows. Let us consider the decomposition of $\Omega$ as $\Omega= \Omega^+\cup\Omega^-$, where $\Omega^+=\{x\in X_0: u(x)\geq0 \}$ and $\Omega^-=\{x\in X_0: u(x)<0\}$ and define $u=u^+-u^-$, where $u^+(x)=\max\{u(x), 0\}$ and $u^-(x)=\max\{-u(x), 0\}$. Let us suppose, $u_n<0$ a.e. in $\Omega$. Now put, $\phi=u^-$ in \eqref{weak p} as the test function and then by using the standard inequality $(a-b)(a^--b^-)\leq-(a^--b^-)^2$, we obtain
%	\begin{align*}
%		&\int_{\Omega}\left(\lambda\frac{u^-}{|u|^{\gamma-1}u}+|u|^{p_s^*-2}uu^-\right)\d x=\mathfrak{M}(\|u\|^p)\int_{Q}\frac{|u(x)-u(y)|^{p-2}(u(x)-u(y))(u^-(x)-u^-(y))}{|x-y|^{N+ps}}\d x\d y\\
%		&\Rightarrow\lambda\int_{\Omega}\frac{sign(u)u^-}{|u|^{\gamma}}\d x\leq-\mathfrak{M}(\|u\|^p)\int_{Q}\frac{|u(x)-u(y)|^{p-2}(u^-(x)-u^-(y))^2}{|x-y|^{N+ps}}\d x\d y\\
%		&\Rightarrow\lambda\int_{\Omega^-}|u^-|^{1-\gamma}dx\leq -\mathfrak{M}(\|u\|^p)\|u^-\|^p<0.
%	\end{align*}
%	Therefore, we can conclude that $|\Omega^-|=0$. This contradicts the assumption $u_n<0$ a.e. in $\Omega$ and hence we guarantee the the solutions to the problem \eqref{main p} are non-negative. Hence the proof.
\end{proof}
\section{Regularity of solutions}
\noindent This section is fully devoted to obtain some regularity of solutions to \eqref{main p}. We begin with the following comparison principle. We borrowed the idea from \cite{Choudhuri2021}.
\begin{lemma}[Weak Comparison Principle]\label{weak comparison}
	Let $u, v\in X_0$, $0<\gamma$. Suppose,\\ $\mathfrak{M}(\|v\|^p)(-\Delta)_{p}^sv-\frac{\lambda}{v^{\gamma}}\geq\mathfrak{M}(\|u\|^p)(-\Delta)_{p}^su-\frac{\lambda}{u^{\gamma}}$ weakly with $v=u=0$ in $\mathbb{R}^N\setminus\Omega$.
	Then $v\geq u$ in $\mathbb{R}^N.$
\end{lemma}
\begin{proof}
	Since, $\mathfrak{M}(\|v\|^p)(-\Delta)_{p}^sv-\frac{\lambda}{v^{\gamma}}\geq\mathfrak{M}(\|u\|^p)(-\Delta)_{p}^su-\frac{\lambda}{u^{\gamma}}$ weakly with $v=u=0$ in $\mathbb{R}^N\setminus\Omega$, we have
	{\small\begin{align}\label{compprinci}
		\langle\mathfrak{M}(\|v\|^p)(-\Delta)_{p}^sv,\phi\rangle-\int_{\Omega}\frac{\lambda\phi}{v^{\gamma}}dx&\geq\langle\mathfrak{M}(\|u\|^p)(-\Delta)_{p}^su,\phi\rangle-\int_{\Omega}\frac{\lambda\phi}{u^{\gamma}}dx~\forall{\phi\geq 0\in X_0}.
		\end{align}}
	
 \noindent We will first prove the comparison result for the Kirchhoff function defined by $\mathfrak{M}(t)=(a+bt^p)\geq a>0$ for $t\geq 0$ and $$\mathcal{M}(t)=\int_{0}^{t}\mathfrak{M}(\tau)d\tau.$$ 
Suppose $S=\{x\in\Omega:u(x)>v(x)\}$ is a set of non-zero measure. Over this set $S$, we have 
% 	In particular choose $\phi=(u-v)^{+}$. To this choice, the inequality in \eqref{compprinci} looks as follows.
\begin{align}\label{compprinci1}
	&(a+b\|v\|^p)(-\Delta)_{p}^sv-(a+b\|u\|^p)(-\Delta)_{p}^su\geq\lambda \left(\frac{1}{v^{\gamma}}-\frac{1}{u^{\gamma}}\right)\geq 0.
	\end{align}
\noindent We will now show that the operator $\mathfrak{M}(\cdot)(-\Delta)_{p}^s(\cdot)$ is a {\it monotone} operator. By the Cauchy-Schwartz inequality we have 
	\begin{eqnarray}\label{csineq}
	|(u(x)-u(y))(v(x)-v(y))|&\leq& |u(x)-u(y)||v(x)-v(y)|\nonumber\\
	& \leq &\frac{|u(x)-u(y)|^2+|v(x)-v(y)|^2}{2}.
\end{eqnarray}
	Consider $I_1=\langle \mathcal{M}'(u),u\rangle-\langle \mathcal{M}'(u),v\rangle-\langle \mathcal{M}'(v),u\rangle+\langle \mathcal{M}'(v),v\rangle$ and let $|u(x)-u(y)| \geq|v(x)-v(y)|$.
	Therefore using \eqref{csineq} we get 
	\begin{align}
	\begin{split}
	I_1=&p \mathfrak{M}(\|u\|^p)\left(\int_{Q}|u(x)-u(y)|^{p-2}\{|u(x)-u(y)|^2-(u(x)-u(y))(v(x)-v(y))\}dxdy\right)\\
	 &+p \mathfrak{M}(\|v\|^p)\left(\int_{q}|v(x)-v(y)|^{p-2}\{|v(x)-v(y)|^2-(u(x)-u(y))(v(x)-v(y))\}dxdy\right)\\
	\geq &\frac{p}{2}  \mathfrak{M}(\|u\|^p)\left(\int_{Q}|u(x)-u(y)|^{p-2}\{|u(x)-u(y)|^2-|v(x)-v(y)|^2\}dxdy\right)\\
	 &+\frac{p}{2}  \mathfrak{M}(\|v\|^p)\left(\int_{Q}|v(x)-v(y)|^{p-2}\{|v(x)-v(y)|^2-|u(x)-u(y)|^2\}dxdy\right)\\
	\geq & p\mathfrak{M}(\|u\|^p)\left(\int_{Q}(|u(x)-u(y)|^{p-2}-|v(x)-v(y)|^{p-2})\right.\\
	&\left.(|u(x)-u(y)|^2-|v(x)-v(y)|^2)dxdy\right).
%	&\geq&p\mathfrak{m}_0\left(\int_{Q}(|u(x)-u(y)|^{p-2}-|v(x)-v(y)|^{p-2})\right.\\
%&\left.(|u(x)-u(y)|%^2-|v(x)-v(y)|^2)dxdy\right).
\end{split}
	\end{align}
	On the other hand, when $|u(x)-u(y)| \leq|v(x)-v(y)|$, we interchange the roles of $u$, $v$ to get 
	\begin{align}
	\begin{split}
	I_1\geq & p\mathfrak{M}(\|u\|^p)\left(\int_{Q}(|u(x)-u(y)|^{p-2}-|v(x)-v(y)|^{p-2})\right.\\
	&\left.(|u(x)-u(y)|^2-|v(x)-v(y)|^2)dxdy\right).
	\end{split}
	\end{align}
	Thus, we deduce that
	\begin{eqnarray}
	\langle \mathcal{M}'(u)-\mathcal{M}'(v),u-v\rangle&=&I_1\geq 0.
	\end{eqnarray}
	Thus $\mathfrak{M}(\cdot)(-\Delta_p)^s(\cdot)$ is a monotone operator. This monotonicity is sufficient for our work.\\
	Coming back to \eqref{compprinci1}, by the monotonicity of $\mathfrak{M}(\cdot)(-\Delta_p)^s(\cdot)$ thus proved implies that $v\geq u$ in $S$. Therefore $u=v$ in $S$ and hence $u\geq v$ a.e. in $\Omega$.
\end{proof}
\noindent We now recall three essential results due to \cite{Brasco2016} to obtain an $L^{\infty}(\bar{\Omega})$ bound. Consider the monotone increasing function $J_{p}(t):=|t|^{p-2}t$ for every $1<p<\infty$.
\begin{lemma}\label{beta convex}
	For every $\beta>0$ and $1\leq p<\infty$ we have
	$$\left(\frac{1}{\beta}\right)^{\frac{1}{p}}\left(\frac{p+\beta-1}{p}\right)\geq 1.$$
\end{lemma}
\begin{lemma}\label{l infty 1}
	Assume $1<p<\infty$ and $f: \mathbb{R}\rightarrow \mathbb{R}$ to be a $C^{1}$ convex function. Then for any $\tau\geq 0$
	\begin{equation}\label{bdd est1}
	J_{p}(a-b)\big[AJ_{p,\tau}(f'(a))-BJ_{p,\tau}(f'(b))\big]\geq(\tau(a-b)^{2}+(f(a)-f(b))^{2})^{\frac{p-2}{2}}(f(a)-f(b))(A-B),
	\end{equation}
	for every $a, b\in \mathbb{R}$ and every $A, B\geq 0$, where $J_{p,\tau}(t)=(\tau+|t|^{2})^{\frac{p-2}{2}}t,~ t\in \mathbb{R}.$ Moreover, for $\tau=0$, we get
	\begin{equation}\label{bdd est1 remark}
	J_{p}(a-b)\big[AJ_{p}(f'(a))-BJ_{p}(f'(b))\big]\geq(f(a)-f(b))^{p-2}(f(a)-f(b))(A-B),
	\end{equation}
	for every $a, b\in \mathbb{R}$ and every $A, B\geq 0$
\end{lemma}
\begin{lemma}\label{l infty 2}
	Assume $1<p<\infty$ and $h:\mathbb{R}\rightarrow \mathbb{R}$ to be an increasing function. Define
	$$G(t)=\int_{0}^{t}h'(\tau)^{\frac{1}{p}}d\tau, t\in \mathbb{R},$$
	then we have
	\begin{equation}\label{bdd est2}
	J_{p}(a-b)(h(a)-h(b))\geq|h(a)-h(b)|^{p}.
	\end{equation}
\end{lemma}
\noindent We now prove the uniform boundedness of solutions to the problem \eqref{main p}. The proof is based on the Moser iteration technique.
\begin{lemma}\label{bounded}
	Let $u\in X_0$ be a positive weak solution to the problem in \eqref{main p}, then $u\in L^{\infty}(\bar{\Omega}).$
\end{lemma}
\begin{proof}
	We will prove this Lemma by obtaining a more general result for any $r\in(p-1,p_s^*]$ in place of the critical exponent, $p_s^*$. The arguments of the proof is taken from the celebrated article of \cite{Brasco2016} with appropriate modifications.	We will proceed with the smooth, convex and Lipschitz function $g_{\varepsilon}(t)=(\varepsilon^2+t^2)^{\frac{1}{2}}$ for every $\varepsilon>0.$ Moreover, $g_{\varepsilon}(t)\rightarrow|t|$ as $t\rightarrow0$ and $|g'_{\varepsilon}(t)|\leq1.$ Let $0<\psi\in C_c^{\infty}(\Omega)$. Choose $\varphi=\psi |g'_{\varepsilon}(u)|^{p-2}g'_{\varepsilon}(u)$ as the test function in \eqref{weak p} with exponent $r\in(p-1,p_s^*]$ in place of $p_s^*$. Now the following estimate follows immediately by putting $a=u(x), b=u(y), A=\psi(x)$ and $B=\psi(y)$ in Lemma \ref{l infty 1}. For all $\psi\in C_c^{\infty}(\Omega)\cap\mathbb{R^+}$, we obtain
	\begin{align}\label{bound est 2.0}
	\mathfrak{M}(\|u\|^p)\int_{Q}&\cfrac{|g_{\varepsilon}(u(x))-g_{\varepsilon}(u(y))|^{p-2}(g_{\varepsilon}(u(x))-g_{\varepsilon}(u(y)))(\psi(x)-\psi(y))}{|x-y|^{N+sp}}dxdy\nonumber\\
	&\leq\int_\Omega\left(\left|\frac{\lambda}{|u|^{\gamma-1}u}+|u|^{r-2}u\right|\right)|g_{\varepsilon}(u)|^{p-1}\psi dx
	\end{align}
	Now on passing to the limit $\varepsilon\rightarrow0$ together with Fatou's Lemma, we obtain
	\begin{align}\label{bound est 2}
	\begin{split}
	\mathfrak{M}(\|u\|^p)\int_{Q}&\cfrac{||u(x)|-|u(y)||^{p-2}(|u(x)|-|u(y)|)(\psi(x)-\psi(y))}{|x-y|^{N+sp}} dxdy\\
	\leq &\int_\Omega\left(\left|\frac{\lambda}{|u|^{\gamma-1}u}+|u|^{r-2}u\right|\right)\psi dx
	\end{split}
	\end{align}
	Note that \eqref{bound est 2} is true for every $\psi\in X_0$. Define $u_k=\min\{(u-1)^+, k\}\in X_0$ for each $k>0$. Let $\beta>0$ and $\delta>0$ be given. Putting $\psi=(u_k+\delta)^{\beta}-\delta^{\beta}$ in \eqref{bound est 2} we get
	\begin{align*}
	\mathfrak{M}(\|u\|^p)\int_{Q}&\cfrac{||u(x)|-|u(y)||^{p-2}(|u(x)|-|u(y)|)((u_k(x)+\delta)^{\beta}-(u_k(y)+\delta)^{\beta})}{|x-y|^{N+sp}}dxdy\\
	&\leq\int_\Omega\left|\frac{\lambda}{|u|^{\gamma-1}u}+|u|^{r-2}u\right|((u_k+\delta)^{\beta}-\delta^{\beta}) dx
	\end{align*}
	On setting $h(u)=(u_k+\delta)^{\beta}$ in Lemma \ref{l infty 2}, we obtain
	\begin{align*}
	&\mathfrak{M}(\|u\|^p)\int_{Q}\cfrac{|((u_k(x)+\delta)^{\frac{\beta+p-1}{p}}
		-(u_k(y)+\delta)^{\frac{\beta+p-1}{p}})|^{p}}{|x-y|^{N+sp}}dxdy\nonumber\\
	&\leq\mathfrak{M}(\|u\|^p)\int_{Q}\left(\cfrac{(\beta+p-1)^{p}}{{\beta}p^{p}}\right)\cfrac{||u(x)|-|u(y)||^{p-2}(|u(x)|-|u(y)|)((u_k(x)+\delta)^{\beta}-(u_k(y)+\delta)^{\beta})}{|x-y|^{N+sp}}dxdy\nonumber\\
	&\leq\left(\cfrac{(\beta+p-1)^{p}}{\beta p^{p}}\right)\mathfrak{M}(\|u\|^p) \int_{Q}\cfrac{||u(x)|-|u(y)||^{p-2}(|u(x)|-|u(y)|)((u_k(x)+\delta)^{\beta}-(u_k(y)+\delta)^{\beta})}{|x-y|^{N+sp}}dxdy\nonumber\\
	&\leq\left(\cfrac{(\beta+p-1)^{p}}{\beta{p}^{p}}\right)\int_{\Omega}\left(\left|\frac{\lambda}{|u|^{\gamma-1}u}\right|+||u|^{r-2}u|\right)\left((u_k+\delta)^{\beta}-\delta^{\beta}\right) dx\nonumber
	\end{align*}
	\begin{align}\label{bound est 4}
	&=\left(\cfrac{(\beta+p-1)^{p}}{\beta{p}^{p}}\right)\left[\int_{\{u\geq1\}}\lambda|u|^{-\gamma}\left((u_k+\delta)^{\beta}-\delta^{\beta}\right)+\int_{\{u\geq1\}}|u|^{r-1}\left((u_k+\delta)^{\beta}-\delta^{\beta}\right) dx\right]\nonumber\\
	&\leq{C}\left(\cfrac{(\beta+p-1)^{p}}{\beta{p}^{p}}\right)\left[\int_{\{u\geq1\}}\left(1+|u|^{r-1}\right)\left((u_k+\delta)^{\beta}-\delta^{\beta}\right) dx\right]\nonumber\\
	&\leq{2C}\left(\cfrac{(\beta+p-1)^{p}}{\beta{p}^{p}}\right)\left[\int_{\Omega}|u|^{r-1}\left((u_k+\delta)^{\beta}-\delta^{\beta}\right) dx\right]\nonumber\\
	&\leq{C'}\left(\cfrac{(\beta+p-1)^{p}}{\beta{p}^{p}}\right){\|u\|^{r-1}_{p_s^*}}\|(u_k+\delta)^{\beta}\|_{q},
	\end{align}
	where $q=\frac{p_s^*}{p_s^*-r+1}$ and $C=\max\{1,|\lambda|\}.$ From the Sobolev inequality \cite{Nezza2012} we get
	\begin{align}\label{bound est 5}
	&\mathfrak{M}(\|u\|^p)\int_{Q}\cfrac{|((u_k(x)+\delta)^{\frac{\beta+p-1}{p}}
		-(u_k(y)+\delta)^{\frac{\beta+p-1}{p}})|^{p}}{|x-y|^{N+sp}}dxdy\geq{C}\mathfrak{m}_0\left\|(u_k+\delta)^{\frac{\beta+p-1}{p}}-\delta^{\frac{\beta+p-1}{p}}\right\|_{p_{s}^*}^{p}
	\end{align}
	where $C>0$. Again from triangle inequality and $(u_k+\delta)^{\beta+p-1}\geq\delta^{p-1}(u_k+\delta)^{\beta}$ we have
	\begin{align}\label{bound est 6}
	\left[\int_{\Omega}\left((u_k+\delta)^{\frac{\beta+p-1}{p}}
	-\delta^{\frac{\beta+p-1}{p}}\right)^{p_s^*}dx\right]^{\cfrac{p}{p_s^*}}\geq\left(\frac{\delta}{2}\right)^{p-1}&\left[\int_{\Omega}(u_k+\delta)^{\frac{p_s^*\beta}{p}}\right]^{\cfrac{p}{p_s^*}}-\delta^{\beta+p-1}|\Omega|^{\cfrac{p}{p_s^*}}.
	\end{align}
	Therefore, by using \eqref{bound est 6} in \eqref{bound est 5} and then from \eqref{bound est 4}, we obtain
	\begin{align}\label{bdd1}
	\left\|(u_k+\delta)^{\frac{\beta}{p}}\right\|^{p}_{p_s^*}
	\leq\frac{C'}{\mathfrak{m}_0}\left[C\left(\frac{2}{\delta}\right)^{p-1}\left(\cfrac{(\beta+p-1)^{p}}{\beta{p}^{p}}\right)\|u\|_{p_s^*}^{r-1}\|(u_k+\delta)^{\beta}\|_{q}+\delta^{\beta}|\Omega|^{\cfrac{p}{p_s^*}}\right].
	\end{align}
	Therefore, on using \eqref{bdd1}, Lemma \ref{beta convex} and \eqref{bound est 4}, we can deduce that
	\begin{align}\label{bound est 7}
	\left\|(u_k+\delta)^{\frac{\beta}{p}}\right\|^{p}_{p_s^*}
	&\leq\frac{C'}{\mathfrak{m}_0}\left[\frac{1}{\beta}\left(\cfrac{\beta+p-1}{p}\right)^{p}\left\|(u_k+\delta)^{\beta}\right\|_{q}\left(\frac{C\|u\|_{p_s^*}^{r-1}}{\delta^{p-1}}+|\Omega|^{\cfrac{p}{p_s^*}-\cfrac{1}{q}} \right)\right].
	\end{align}	
	\noindent Now choose, $\delta>0$ such that $\delta^{p-1}=C\|u\|_{p_s^*}^{r-1}\left(|\Omega|^{\frac{p}{p_s^*}-\frac{1}{q}}\right)^{-1}$ and $\beta\geq1$ with $\left(\frac{\beta+p-1}{p}\right)^{p}\leq\beta^{p}.$ Further, by setting $\eta=\cfrac{p_s^*}{pq}>1$ and $\tau=q\beta$ we can rewrite the inequality \eqref{bound est 7} as
	\begin{align}\label{bound est 8}
	\left\|(u_k+\delta)\right\|_{\eta\tau}\leq\left(C|\Omega|^{\frac{p}{p_s^*}-\frac{1}{q}}\right)^{\frac{q}{\tau}}\left(\frac{\tau}{q}\right)^{\frac{q}{\tau}}\left\|(u_k+\delta)\right\|_{\tau}.
	\end{align}
	Set $\tau_0=q$ and $\tau_{m+1}=\eta\tau_m=\eta^{m+1}q$. Then after performing $m$ iterations, the inequality \eqref{bound est 8} reduces to
	\begin{align}\label{bound est 9}
	\left\|(u_k+\delta)\right\|_{\tau_{m+1}}&\leq\left(C|\Omega|^{\frac{p}{p_s^*}-\frac{1}{q}}\right)^{\left(\sum\limits_{i=0}^{m}\frac{q}{\tau_i}\right)}\left(\prod\limits_{i=0}^{m}\left(\frac{\tau_i}{q}\right)^{\frac{q}{\tau_i}}\right)^{p-1}\left\|(u_k+\delta)\right\|_{q}\nonumber\\
	&=\left(C|\Omega|^{\frac{p}{p_s^*}-\frac{1}{q}}\right)^{\frac{\eta}{\eta-1}}\left(\eta^{\frac{\eta}{(\eta-1)^2}}\right)^{p-1}\left\|(u_k+\delta)\right\|_{q}
	\end{align}
	Therefore, on passing the limit as $m\rightarrow\infty$, we get
	\begin{equation}\label{bound est 10}
	\left\|u_k\right\|_{\infty}\leq\left(C|\Omega|^{\frac{p}{p_s^*}-\frac{1}{q}}\right)^{\frac{\eta}{\eta-1}}\left(C'\eta^{\frac{\eta}{(\eta-1)^2}}\right)^{p-1}\left\|(u_k+\delta)\right\|_{q}.
	\end{equation}
	 Furthermore, by applying the triangle inequality together with the fact $u_k\leq(u-1)^+$ in \eqref{bound est 10} and then letting $k\rightarrow\infty$, we obtain
	\begin{equation}\label{bound est 11}
	\left\|(u-1)^+\right\|_{\infty}\leq\left\|u_k\right\|_{\infty}\leq{C}\left(\eta^{\frac{\eta}{(\eta-1)^2}}\right)^{p-1}\left(|\Omega|^{\frac{p}{p_s^*}-\frac{1}{q}}\right)^{\frac{\eta}{\eta-1}}\left(\left\|(u-1)^+\right\|_q+\delta|\Omega|^{\frac{1}{q}}\right)
	\end{equation}
	Hence, we have $u\in L^{\infty}({\Omega}).$ In particular, by choosing $r=p_s^*$, we conclude the that if $u\in X_0$ is a solution to \eqref{main p}, then $u\in L^{\infty}(\bar{\Omega}).$
\end{proof}
\section*{Appendix}

We now prove the existence of the first eigenvalue and eigenfunction for the fractional $p$-Kirchhoff equation in the form of the following Lemma.
\begin{lemma}\label{first_eig_value}
{Let $(\lambda_1,\phi_1)$ be the first eigenpair to the problem
\begin{align}\label{p-lap-eig-val-1}
\begin{split}
(-\Delta)_p^s\phi_1=&\lambda_1|\phi_1|^{p-2}\phi_1,~\text{in}~\Omega\\
\phi_1>&0,~\text{in}~\Omega\\
\phi_1=&0,~\text{in}~\mathbb{R}^N\setminus\Omega.
\end{split}
\end{align}
Then $(\Lambda_1,u_1)$ is the first
eigenpair of 
\begin{align}\label{p-lap-eig-val-2}
\begin{split}
\mathfrak{M}(\|u\|^p)(-\Delta)_p^su_1=&\lambda_1|u_1|^{p-2}u_1,~\text{in}~\Omega\\
u_1>&0,~\text{in}~\Omega\\
u_1=&0,~\text{in}~\mathbb{R}^N\setminus\Omega.
\end{split}
\end{align} 
where $\Lambda_1=\mathfrak{M}(t^p\|\phi_1\|^p)\lambda_1$, $u_1=t\phi_1$ for some $t>0$.}
\end{lemma}
\begin{proof}
Taking $u_1=t\phi_1$, $t > 0$ in \eqref{p-lap-eig-val-1} yields that $\Lambda_1=\mathfrak{M}(t^p\|\phi_1\|^p)\lambda_1$. This completes the proof.
\end{proof}

\begin{remark}\label{dj1}
{
\begin{enumerate}
\item The function $u_1\in L^{\infty}(\Omega)$ since $\phi_1\in L^{\infty}(\Omega)$. 
\item Let $\underline{u}$ be the solution to the problem \eqref{squasinna}. For $\epsilon>0$ sufficiently small, consider 
\begin{align}\label{dj2}
\begin{split}
\mathfrak{M}(\|\epsilon u_1\|^p)(-\Delta)_p^s(\epsilon u_1)-\lambda\epsilon^{-\gamma}{u_1}^{-\gamma}<0=\mathfrak{M}(\|\underline{u}\|^p)(-\Delta)_p^s(\underline{u})-\lambda {\underline{u}}^{-\gamma}.
\end{split}
\end{align}
By the Lemma \ref{weak comparison} we have $\underline{u}\geq\epsilon u_1$.
\end{enumerate}}
\end{remark}
%\textcolor{red}{It may be observed from \cite{Saoudi2019} that $c_2 %d(x,\partial\Omega)\geq\phi_1\geq c_1 d(x,\partial\Omega)$ and hence %$\underline{u}\geq\epsilon u_1$.}

%\begin{remark}\label{cut-off-fnal-pf}
%	Let $A=\{x\in\Omega:\underline{u}(x)-u(x)>0\}$. Then we have
%	\begin{align}\label{simon_use_0}
%	\begin{split}
%	\langle \mathfrak{M}(\|u\|^p)(-\Delta)_p^su-\mathfrak{M}(\|\underline{u}\|^p)(-\Delta)_p^s\underline{u},\underline{u}-u\rangle_A&\geq 0~\text{by the Simon's inequality}
%	\end{split}
%	\end{align}
%	where $\langle\cdot,\cdot\rangle_A$ denotes the integration to be over the measurable subset $A$ of $\Omega$. Thus, if
%	\begin{align}\label{simon_use_1}
%	\begin{split} \mathfrak{M}(\|u\|^p)(-\Delta)_p^su&\geq\mathfrak{M}(\|\underline{u}\|^p)(-\Delta)_p^s\underline{u}~\text{in}~A.
%	\end{split}
%	\end{align}

%\end{remark}

\section*{Data Availability statement}
\noindent The article has no data associated to it whatsoever.
\section*{Conflict of interest statement}
\noindent There is no conflict of interest whatsoever.
\section*{Acknowledgement}
\noindent The author S. Ghosh thanks the Council of Scientific and Industrial Research (CSIR), India, for the financial assistantship received to carry out this research work. D. Choudhuri thanks the National Board for Higher Mathematics (N.B.H.M.), Department of Atomic Energy (DAE) India, [02011/47/2021/NBHM(R.P.)/R\&D II/2615] India,to carry out the research. A.\,Fiscella is member of the {Gruppo Nazionale per l'Analisi Ma\-tema\-tica, la Probabilit\`a e
	le loro Applicazioni} (GNAMPA) of the {Istituto Nazionale di Alta Matematica ``G. Severi"} (INdAM).
A.\,Fiscella realized the manuscript within the auspices of the FAPESP Thematic Project titled ``Systems and partial differential equations" (2019/02512-5).

\end{document}